\documentclass[12pt,a4paper]{article}
\usepackage{amsmath,amsfonts,amssymb,amsthm,latexsym}
\usepackage{amscd,graphicx,stmaryrd}
\usepackage{cite}\usepackage{mathrsfs}
\usepackage[all]{xy}
\begin{document}
\renewcommand{\baselinestretch}{1.025}
\newcommand{\nc}[2]{\newcommand{#1}{#2}}
\newcommand{\coc}{{\mathrm{co}C}}

\newtheorem{prop}{Proposition}[section]
\newtheorem{cor}[prop]{Corollary}
\newtheorem{thm}[prop]{Theorem}
\newtheorem{lem}[prop]{Lemma}
\newtheorem{dfn}[prop]{Definition}
\newtheorem{exam}[prop]{Example}
\newtheorem{rem}[prop]{Remark}

\newcommand{\rnc}[2]{\renewcommand{#1}{#2}}
\rnc{\[}{\begin{equation}}
\rnc{\]}{\end{equation}}
\nc{\wegengruen}{\end{equation}}
\numberwithin{equation}{section}

\def\CC{C}
\def\C{{\mathbb C}}
\def\N{{\mathbb N}}
\def\R{{\mathbb R}}
\def\Z{{\mathbb Z}}
\def\cO{{\mathcal O}}
\def\cK{{\mathcal K}}
\def\H{{\mathscr H}}
\def\cH{{\mathcal H}}
\def\cP{{\mathcal P}}
\def\T{{\mathcal T}}
\def\B{{\mathcal B}}
\def\I{{\mathcal I}}
\def\P{\mathcal{P}_{\mathrm{U}(1)}}
\def\PW{P}
\def\lP{P}
\def\PWG{\mathcal{P}_{\!\delta}}
\def\PWU{\mathcal{P}_{\!\Delta}}
\def\barP{\bar P}
\nc{\ha}{\mbox{$\alpha$}}
\nc{\hb}{\mbox{$\beta$}}
\nc{\hg}{\mbox{$\gamma$}}\nc{\hc}{\mbox{$\gamma$}}
 \nc{\hd}{\mbox{$\delta$}}
\nc{\he}{\mbox{$\varepsilon$}}
\nc{\hz}{\mbox{$\zeta$}}
\nc{\s}{\mbox{$\sigma$}}
\nc{\hk}{\mbox{$S$}}     
\nc{\hm}{\mbox{$\mu$}}
\nc{\hn}{\mbox{$\nu$}}
\nc{\hl}{\mbox{$\lambda$}}
\nc{\hG}{\mbox{$\Gamma$}}
\nc{\hD}{\mbox{$\Delta$}}
\nc{\hT}{\mbox{$\Theta$}}
\nc{\ho}{\mbox{$\omega$}}
\nc{\hO}{\mbox{$\Omega$}}
\nc{\hp}{\mbox{$\pi$}}
\nc{\hP}{\mbox{$\Pi$}} 
\nc{\vare}{\mbox{$\varepsilon$}}

\newcommand{\co}{{\mathrm{co}}}
\renewcommand{\S}{{s}}
\newcommand{\U}{u}
\newcommand{\V}{u}
\newcommand{\K}{\mathrm{K}}
\newcommand{\ip}[1]{\langle#1\rangle}

\nc{\hs}{\hspace{1pt}} \nc{\hsp}{\hspace{-1pt}}
\nc{\SUq}{\mbox{$\cO(\mathrm{SU}_q(2))$}}
\nc{\CSUq}{\mbox{$\CC(\mathrm{SU}_q(2))$}}
\nc{\OU}{\mbox{$\cO(\mathrm{U}(1))$}}
\nc{\rmU}{\mbox{$\mathrm{U}(1)$}}
\def\Sq{\mbox{$\cO(\mathrm{S}_{q}^2)$}}
\def\Dq{\mbox{$\cO(\mathrm{D}_{q})$}}
\def\rmD{\mbox{$\mathrm{D}$}}
\def\CSq{\mbox{$\CC(\mathrm{S}_{q}^2)$}}
\def\S{\mbox{$\mathrm{S}$}}
\nc{\lN}{{\ell}_2(\N)}
\def\im{{\mathrm{i}}}
\def\CS{\mbox{$\CC(\mathrm{U}(1))$}}
\def\lra{\longrightarrow}
\def\ra{\rightarrow}
\def\id{{\mathrm{id}}}
\def\Mat{{\mathrm{Mat}}}
\def\tr{{\mathrm{Tr}}}
\def\ker{{\mathrm{ker}}}
\nc{\lin}{\mbox{$\mathrm{span}$}}
\nc{\pink}{\mbox{$\Pi_{k=1}^{n}$}}
\newcommand{\can}{\mathsf{can}}
\newcommand{\tcan}{\widetilde{\mathsf{can}}}
\newcommand{\hrho}{\hat{\rho}}
\newcommand{\ov}{\overline}
\newcommand{\srep}{*-re\-pre\-sen\-ta\-tion}
\nc{\inc}{\subseteq}
\def\ot{\otimes}
\newcommand{\LAblp}{\mbox{\LARGE\boldmath$($}}
\newcommand{\LAbrp}{\mbox{\LARGE\boldmath$)$}}
\newcommand{\Lblp}{\mbox{\Large\boldmath$($}}
\newcommand{\Lbrp}{\mbox{\Large\boldmath$)$}}
\newcommand{\lblp}{\mbox{\large\boldmath$($}}
\newcommand{\lbrp}{\mbox{\large\boldmath$)$}}
\newcommand{\blp}{\mbox{\boldmath$($}}
\newcommand{\brp}{\mbox{\boldmath$)$}}
\newcommand{\LAlp}{\mbox{\LARGE $($}}
\newcommand{\LArp}{\mbox{\LARGE $)$}}
\newcommand{\Llp}{\mbox{\Large $($}}
\newcommand{\Lrp}{\mbox{\Large $)$}}
\newcommand{\llp}{\mbox{\large $($}}
\newcommand{\lrp}{\mbox{\large $)$}}
\newcommand{\Cc}{c}
\newcommand{\Dd}{d}
\newcommand{\pr}{\mathrm{pr}}
\newcommand{\barot}{\hs\bar\otimes\hs}

\hyphenation{equi-va-ri-ance equi-va-ri-ant geo-me-try ope-ra-tor
ope-ra-tors Pod-le}
%

\title{\vspace*{-15mm}\Large\bf
THE PULLBACKS OF PRINCIPAL COACTIONS}
\author{
\vspace*{-0mm}\large\sc
Piotr M.~Hajac\\
\vspace*{-0mm}\large
Instytut Matematyczny, Polska Akademia Nauk\\
\vspace*{-0mm}\large
ul.\ \'Sniadeckich 8, Warszawa, 00-956 Poland
\vspace{0mm}\\
\large\sl
http://www.impan.pl/$\!\widetilde{\phantom{m}}\!$pmh\\
\vspace*{0mm}\large
and\\
\vspace*{-0mm}\large
Katedra Metod Matematycznych Fizyki, Uniwersytet Warszawski\\
\vspace*{-0mm}\large
ul.\ Ho\.za 74, Warszawa, 00-682 Poland\\
\vspace*{0mm}\and
\vspace*{-0mm}\large\sc
Elmar Wagner\\
\vspace*{-0mm}\large
Instituto de F\'isica y Matem\'aticas\\
\vspace*{-0mm}\large
Universidad Michoacana de San Nicol\'as de Hidalgo\\
\vspace*{-0mm}\large
 Edificio C-3, Cd.~Universitaria,
58040 Morelia, Michoac\'an, M\'exico\\
\vspace{-0mm}\large\sl
e-mail: elmar@ifm.umich.mx 
}
\date{}
\maketitle

{\vspace*{-1mm}\noindent\small {\bf Abstract:}
We prove that the class of principal coactions is closed under one-surjective
pullbacks in an appropriate category of algebras equipped with left and right coactions.
This allows us to handle cases of $C^*$-algebras lacking two different non-trivial ideals. It also allows us to go beyond the category of
comodule algebras.
As an example of the former, we carry out an index computation for noncommutative line bundles over
the standard Podle\'s sphere using the Mayer-Vietoris type arguments afforded by
a one-surjective pullback presentation of the $C^*$-algebra of this quantum sphere. To instantiate the latter, we define a family of
coalgebraic noncommutative  deformations of the $\mathrm{U}(1)$-principal bundle 
$\mathrm{S}^7\rightarrow\mathbb{C}\mathrm{P}^3$.
}
{\vspace*{-2mm}\footnotesize\tableofcontents}

\normalsize
\section{Introduction and preliminaries}

The idea of decomposing a complicated object into simpler pieces and
connecting data is a fundamental computational principle throughout
mathematics. In the case of (co)homology theory, it yields the 
Mayer-Vietoris long exact sequence whose significance and usefulness
can hardly be overestimated. The categorical underpinning of all this
are pullback diagrams: in a given category they give a rigorous meaning
to putting together two objects over a third one.

The goal of this paper is to prove a general pullback theorem for 
principal coactions that
significantly generalizes the main result of \cite{hkmz}
restricted to comodule algebras and pullbacks of surjections.
More precisely, our main result is that the pullback of principal
coactions over morphisms of which at least one is surjective is again
a principal coaction. It may be viewed as a non-linear version
 of the Milnor construction yielding an odd-to-even connecting 
homomorphism in algebraic $K$-theory~\cite{mi}.
 Indeed, linearizing our pullback theorem
with the help of a corepresentation gives
precisely the odd-to-even construction of a projective module defining
 the connecting homomorphism in $K$-theory. 

We apply this new result in two cases. In the first case, we keep
the comodule-algebra setting but take a one-surjective pullback diagram
(only one of the defining morphisms is surjective). In the second case,
we proceed the other way round, that is, we take a pullback diagram
given by two surjections but take coactions that are not algebra
homomorphisms.

The pullback picture of the standard quantum Hopf fibration gives
us the first-case example. It provides a new way of computing the 
index pairing for the associated quantum Hopf line bundles 
(cf.~\cite{w-e}). This index pairing  was computed
in \cite{h-pm00} using a noncommutative index formula, and
re-derived in~\cite{nt05}. Here we give yet another method to compute it.
This simple example shows the need to generalize from two-surjective
to one-surjective pullback diagrams, and the pullback
method of index computation seems attractive due to its inherent
simplicity.

To obtain the second-case example, we first show how the piecewise
structure \cite{hkmz} 
of a noncommutative join construction \cite{dhh} allows one to define
a certain  class of piecewise principal coactions. 
Although this class of examples can also be handled by earlier methods,
it definitely shows that there are interesting piecewise principal
coactions that are not algebra homomorphisms.
To obtain a concrete example, we take
 Pflaum's instanton bundle $\mathrm{S}^7_q\rightarrow
\mathrm{S}^4_q$ \cite{p-mj94} as 
the noncommutative join 
 of $\mathrm{SU}_q(2)$  and turn it into the coalgebraic
quantum principal bundle \mbox{$\mathrm{S}^7_q\rightarrow \C \mathrm{P}^3_{q,s}$}.
We do it with the help
of the canonical surjections 
$\pi :\mathcal{O}(\mathrm{SU}_q(2))\rightarrow 
\mathcal{O}(\mathrm{SU}_q(2))/J_{q,s}$
determined by the coideals right ideals
$J_{q,s}:=
(\cO(\mathrm{S}_{q,s}^2)\cap \mathrm{ker}\hs \vare)\mathcal{O}(\mathrm{SU}_q(2))$,
where $\mathrm{S}_{q,s}^2$ is a generic Podle\'s quantum sphere
\cite{p-p87} and $\mathrm{ker}\hs\vare$ is the kernel of the counit map.
\\

The paper is organized as follows. First, to make our exposition 
self-contained and to establish notation, we recall fundamental
concepts that we use later on. The key Section~2 is devoted to
the general pullback theorem for principal coactions of coalgebras
on algebras,  Section~3 is on deriving the index pairing
 for quantum Hopf line bundles as a corollary to the pullback
 presentation of the standard Hopf fibration of~$\mathrm{SU}_q(2)$,
and the final Section~4 presents new examples of piecewise principal
coactions that go beyond Hopf-Galois theory.

Throughout the paper, we work with algebras and coalgebras over a field.
The unadorned tensor product stands for the algebraic tensor product over this field. 
We employ the Heyneman-Sweedler type notation 
(with the summation symbol suppressed) for the comultiplication
$\Delta(c)\!=\!c_{(1)}\otimes c_{(2)}\in C\otimes C$
and for coactions
$\Delta_V(v)=v_{(0)}\otimes v_{(1)}\in V\otimes C$,
${}_V\Delta(v)=v_{(-1)}\otimes v_{(0)}\in C\otimes V$. The convolution
product of two linear maps from a coalgebra to an algebra is denoted
by $*$: $(f*g)(c):=f(c_{(1)})g(c_{(2)})$. The set of natural numbers 
includes $0$, that is, $\N=\{0,1,2,\ldots \}$.

\subsection{Pullback diagrams and fibre products}\label{sec-fp}

The purpose of this section is to collect some elementary facts about 
fibre products. We  consider the category 
of vector spaces as it will be the ambient category for all
our pullback diagrams. Let 
$\pi_1: A_1 \rightarrow A_{12}$ and  $\pi_2:A_2\rightarrow
A_{12}$ be linear maps. The {\em fibre product} of these
maps is defined by
\begin{equation}\label{A}
A_1\underset{(\pi_1,\pi_2)}{\times} A_2:=
\left\lbrace (a_1,a_2)\in A_1\times A_2 \;|\; \pi_1(a_1)=\pi_2(a_2)
\right\rbrace.
\end{equation}
Together with the canonical projections
\begin{equation}
\mathrm{pr}_1:A_1\underset{(\pi_1,\pi_2)}{\times} A_2\longrightarrow A_1,
\qquad
\mathrm{pr}_2:A_1\underset{(\pi_1,\pi_2)}{\times} A_2\longrightarrow A_2
\end{equation}
it forms a universal construction completing the  initially
given two linear maps into
 the following commutative diagram:
\begin{equation}\label{A_is_fibre_product}
        \begin{CD}
    {A_1\underset{(\pi_1,\pi_2)}{\times} A_2} @ >{\mathrm{pr}_2}>> {A_2} @.\\
    @ V{\mathrm{pr}_1} VV @ V{\pi_2} VV @.\\
    {A_1} @ >{\pi_1} >> {A_{12}} @. \ .\\
        \end{CD}
\end{equation}
Such  diagrams are called {\em pullback diagrams},
and fibre products are often referred to as pullbacks.  

Next, if $\pi_1: A_1 \rightarrow A_{12}$ and  $\pi_2:A_2\rightarrow
A_{12}$ are morphisms of *-algebras, then the fibre product 
$A_1{\times}_{(\pi_1,\pi_2)} A_2$ is a *-subalgebra
of $A_1\times A_2$.
Furthermore, if we consider the pullback diagram 
\eqref{A_is_fibre_product} in
the category of (unital) $C^*$-algebras, then 
$A_1{\times}_{(\pi_1,\pi_2)} A_2$
with its componentwise multiplication
and *-structure is a (unital) $C^*$-algebra.
Much the same, if $B$ is an algebra and $\pi_1: A_1 \rightarrow
A_{12}$ and  $\pi_2:A_2\rightarrow A_{12}$ are morphisms of left
$B$-modules, then the fibre product
$A_1{\times}_{(\pi_1,\pi_2)} A_2$ is a left $B$-module
via the componentwise left action $b.(a_1,a_2)=(b. a_1,b. a_2)$.

\subsection{Odd-to-even connecting homomorphism in $K$-theory}\label{sec-mv}

Consider a pullback diagram  
\begin{equation}
\mbox{$\xymatrix@=5mm{& & A \ar[lld]
\ar[rrd] & &\\
A_1 \ar@{>>}[drr]_{\pi_1}& & & &A_2 \ar[dll]^{\pi_2}\\
&& A_{12} &&}$}
\end{equation}
in the category of unital
algebras, and assume that  one of the defining
morphisms (here we choose~$\pi_1$) is surjective.
Then there exists a
long exact sequence in algebraic $K$-theory \cite{mi}
\begin{equation} 
\xymatrixcolsep{4pc}\xymatrix{
\cdots\longrightarrow {\K}_1(A_{12}) \  
\ar[r]^{\mbox{ }\hspace{-64pt}\text{odd-to-even}} 
& \ 
{\K}_0(A)\longrightarrow {\K}_0(A_1\oplus A_2)\longrightarrow {\K}_0 (A_{12}).
}
\end{equation}

The mapping 
$\xymatrix{{\K}_1(A_{12})  
\ar[rr]^{\mbox{ }\hspace{-2pt}\text{odd-to-even}} & & {\K}_0(A) }$ 
is obtained as follows. 
First, given left $A_i$-modules $E_i$, $i=1,2$, we obtain
  left $A_{12}$-modules
$\pi_{i\ast}E_i$ defined by $A_{12}\otimes_{A_i} E_i$. 
Since $A_{12}$ is unital,
there are canonical morphisms $\pi_{i\ast}:E_i\ra\pi_{i\ast}E_i$,
$\pi_{i\ast}(e)=1\otimes_{A_i} e$.
The modules $E_i$ and $\pi_{i\ast}E_i$ can also be  considered
as  left modules over the fibre product algebra
$A$ via the left actions given by
$a.e_i= {\mathrm{pr}_i}(a).e_i$, for $e_i\in E_i$, and
$a.f_i= \pi_i({\mathrm{pr}_i}(a)).f_i$, for $f_i\in \pi_{i\ast}E_i$.
Assume now that
$h:\pi_{1\ast}E_1\ra\pi_{2\ast}E_2$ is a morphism
of left $A_{12}$-modules.
Then $h\circ\pi_{1\ast}: E_1\ra\pi_{2\ast}E_2$ and
$\pi_{2\ast}: E_2\ra\pi_{2\ast}E_2$ can be lifted to morphisms
of left $A$-modules, and we can consider their pullback
diagram in the category of left $A$-modules:
\begin{equation}\label{fpMod}
\xymatrix{
& E_1 \,{\underset{(h\circ\pi_{1\ast} ,\pi_{2\ast})}{\times}}\, 
E_2\ar[dl]_{\mathrm{pr}_1} \ar[dr]^{\mathrm{pr}_2}& \\
E_1 \ar[d]_{\pi_{1\ast}} &    &  E_2 \ar[d]^{\pi_{2\ast}}\\
\pi_{1\ast}E_1\ar[rr]_{h} & & \pi_{2\ast}E_2\,. }
\end{equation}

In \cite[Section 2]{mi}, it is proven in detail that, if $E_i$ 
is a finitely generated projective module over $A_i$, $i=1,2$,
and $h$ is an isomorphism, 
then the fibre product 
$M:= E_1\times_{(h\circ\pi_{1\ast}\hs,\,\pi_{2\ast})}E_2$ 
is a finitely generated projective $A$-module. 
Furthermore, up to isomorphism, every finitely generated
projective module over $A$ has this form, and the $A_i$-modules 
$E_i$ and \mbox{${\mathrm{pr}_{i\ast}}M:=A_i\otimes_A M$}, $i=1,2$, 
are naturally isomorphic. In particular, if
$E_1\cong A_1^n$ and $E_2\cong A_2^n$, 
the isomorphism 
$h:\pi_{1\ast}E_1\ra\pi_{2\ast}E_2$ is given by an 
invertible matrix $U\in \mathrm{GL}_n(A_{12})$. 
Using the canonical embedding 
$\mathrm{GL}_n(A_{12})\subseteq\mathrm{GL}_\infty(A_{12})$, 
we get a map 
\begin{equation}\label{UM}
 \mathrm{GL}_\infty(A_{12})\ni U\longmapsto M\in \mathrm{Proj}(A)
\end{equation}
given by the pullback diagram 
\begin{equation}
\mbox{$\xymatrix@=5mm{& & M \ar[lld]
\ar[rrd] & &\\
A_1^n \ar@{>>}[drr]_{\pi_1}& & & &A_2^n \ar[dll]^{\pi_2}\\
&& A_{12}^n\stackrel{U}{\cong}A_{12}^n\ . &&}$}
\end{equation}

This map induces 
 an odd-to-even connecting homomorphism on the level of 
both algebraic \cite{mi} and $C^*$-algebraic \cite{hrz}
$K$-theory.
An explicit description of the module $M$ is as follows. 
Assume that $\pi_1:A_1\ra A_{12}$ is surjective.
Then there exist liftings ${\Cc},{\Dd}\in\Mat_n(A_{1})$ such
that evaluating $\pi_1$ on ${\Cc}$ and $\Dd$ componentwise
yields $U^{-1}$ and $U$ respectively.
Applying \cite[Theorem~2.1]{DHHM} to our situation yields 
$E_1\times_{(h\circ\pi_{1\ast}\hs,\,\pi_{2\ast})}E_2\cong A^{2n}p$,
where
\begin{equation}\label{p}
 p=
  \left(
  \begin{array}{cc}
  ({\Cc}(2-{\Dd}{\Cc}){\Dd},1 ) & ({\Cc}(2-{\Dd}{\Cc})(1-{\Dd}{\Cc}),0 ) \cr
  ((1-{\Dd}{\Cc}){\Dd},0 ) & ((1-{\Dd}{\Cc})^2,0  )
  \end{array}
  \right)\in \Mat_{2n}(A).
 \end{equation}

\subsection{Principal coactions and associated projective modules}
\label{Pe}

Recall first the general definition of an entwining structure.
Let $C$ be a coalgebra with  comultiplication $\hD$ and 
counit $\he$, and let $A$ be an algebra with  multiplication $m$ 
and  unit $\eta$. 
A~linear map
\begin{equation}
\psi:C\otimes A\lra A\otimes C
\end{equation}
is called an {\em  entwining structure} if and only if it is unital,
counital, and distributive
with respect to both the multiplication and comultiplication:
\begin{align}\label{entwining}      
&\psi\circ(\id\ot m)=
(m\ot \id)\circ (\id \ot \psi)\circ (\psi \ot \id),
\qquad \psi\circ(\id\ot\eta)=(\eta\ot\id)\circ\text{\rm flip},\\
\label{e2}
&(\id\ot\Delta)\circ\psi=
(\psi \ot \id)\circ (\id \ot \psi)\circ (\Delta\ot \id),
\qquad (\id\ot\he)\circ\psi=\text{\rm flip}\circ(\he\ot\id) .
\end{align}
If $\psi$ is an entwining of a coalgebra $C$ and an algebra $A$, 
and $M$ is
a right $C$-comodule and a right $A$-module, 
we call $M$ an \emph{entwined
module} \cite{b-t99} when
it satisfies the compatibility condition
\[\label{entmod}
(ma)_{(0)}\otimes (ma)_{(1)}=m_{(0)}\psi(m_{(1)}\otimes a).
\]

Next, 
let $P$ be an algebra equipped with a coaction $\Delta_P:P\ra P\ot C$
of a coalgebra $C$.
Define the coaction-invariant subalgebra of $P$ by
\begin{equation}\label{coinvar}
B:=P^{\co C}:=
\{b\in P \;|\; \Delta_P(bp)=b\Delta_P(p)\ \,\text{for all}\ \hs p\in P\}.
\end{equation}
We call the inclusion $B\inc P$ a $C$-extension.
We call it a {\em coalgebra-Galois $C$-extension} 
 when the
canonical left $P$-module right $C$-comodule map
\begin{equation}  \label{can}
\can: P\underset{B}{\ot}P{\lra} P\ot C,\quad p\underset{B}{\ot}
 p'\longmapsto p\hD_P(p'),
\end{equation}
is bijective~\cite{bh99}. Note that the bijectivity of 
$\can$ allows us to define the so-called
translation map 
\[
 \tau : C\lra P \underset{B}{\otimes} P, \quad \tau(c) 
:= \can^{-1}(1\ot c). 
\]
Moreover, every coalgebra-Galois $C$-extension
comes naturally equipped with a unique entwining structure
that makes $P$ a $(P,C)$-entwined module in the sense of~\eqref{entmod}. 
It is called the
canonical entwining structure~\cite{bh99}, and is very useful
in calculations or further constructions. 
Explicitly, it can be written as: 
\[       \label{can-ent}                         
\psi(c\otimes p)= \can(\can^{-1}(1\otimes c)p).
\]

An algebra $P$ with a right $C$-coaction $\hD_P$  
is said to be {\em $e$-coaugmented}
if and only if there exists a group-like element $e \in C$ such that
$\Delta_P (1) = 1 \otimes e$. We call the $C$-extension
\mbox{$B:=P^{\co C}\inc P$} $e$-coaugmented.
(Much the same way, one
defines the coaugmentation of left coactions.)
For the $e$-coaugmented coalgebra-Galois $C$-extensions, 
one can show that the coaction-invariant
subalgebra defined in \eqref{coinvar} can be expressed as 
\begin{equation}\label{ecoinvar}
P^{\co C}=\{p\in P\;|\;\hD_P(p)=p\ot e\}.
\end{equation}
Indeed, Formula~\eqref{can-ent}
allows us to express the right coaction in terms of the entwining
\[\label{iii}
\hD_P(p)=\psi(e\ot p),
\] 
and Equation~\eqref{entwining} yields the right-in-left inclusion. 
The opposite inclusion is obvious.

Next, if $\psi$ is invertible, 
one can use \eqref{e2} to show that the formula 
\begin{equation}\label{leftco}
_P\hD(p):=\psi^{-1}(p\ot e)
\end{equation}
defines a left coaction $_P\hD:P\ra C\ot P$. 
We define the left coaction-invariant subalgebra ${}^{\co C}\!P$ as in 
\eqref{coinvar}, and derive the left-sided version of \eqref{can-ent}. 
Hence, for any $e$-coaugmented coalgebra-Galois $C$-extension
 with {\em invertible canonical entwining},  the right coaction-invariant
subalgebra  coincides with the left coaction-invariant subalgebra:
\[  \label{lr}
 P^{\co C}=\{p\in P\;|\;\hD_P(p)=p\ot e\}
 =\{p\in P\;|\;{}_P\hD(p)=e\ot p\}={}^{\co C}\!P.
\]

Finally, we need to assume  one more condition on $C$-extensions to obtain
a suitable definition of a principal coaction:  
{\em equivariant projectivity}. 
It is a pivotal property that
guarantees the projectivity  of associated modules, 
and thus leads to index pairings
between $K$-theory and $K$-homology. 
Putting together the aforementioned four
conditions, we say that a coalgebra
  $C$-extension $B\inc P$ is {\em principal}~\cite{bh04}
if:
\begin{enumerate}
\item[(i)] The canonical map $\can: P\ot_BP{\ra} P\ot C$, 
$p\ot_B p'\mapsto p\hD_P(p')$, 
is bijective 
(Galois condition).
\item[(ii)] The right coaction is $e$-coaugmented for some group-like 
$e\in C$, i.e.\ $\Delta_P(1)=1\ot e$.
\item[(iii)]
The canonical entwining
$\psi:C\otimes P{\ra} P\ot C$, $c\otimes p\mapsto
\can(\can^{-1}(1\otimes c)p)$,  is bijective.
\item[(iv)]
The algebra $P$ is $C$-equivariantly projective as a left $B$-module,
i.e.\ there exists a  left $B$-linear and right $C$-colinear
splitting of the multiplication map
$B \otimes P \rightarrow P$.
\end{enumerate}

In the framework of coalgebra extensions, the role 
of connections on principal bundles 
is played by strong connections~\cite{bh04}. 
Let $P$ be an algebra and both a left and right $e$-coaugmented $C$-comodule. 
(Note that the left and right coactions need not commute.)
A \emph{strong connection} is  a  linear map
$\ell : C \rightarrow P \otimes P$
satisfying
\begin{equation}\label{sc}
\tcan \circ \ell\!=\!1 \otimes \id,
\quad
(\mathrm{id} \otimes \Delta_P) \circ
\ell \!=\! (\ell \otimes \mathrm{id}) \circ \Delta,
\quad
({}_P \Delta \otimes \mathrm{id}) \circ
\ell\! =\! (\mathrm{id} \otimes \ell) \circ
\Delta,
\quad
\ell(e)=1\ot 1.
\end{equation}
Here
$\tcan : P \otimes P \rightarrow P \otimes C$
is the  lifting of $\can$ to $P \otimes P$.
Assuming that there exists an invertible entwining $\psi: C\ot P\ra P\ot C$ making $P$ an entwined module, 
the first three equations of \eqref{sc} read 
in the Heyneman-Sweedler type notation 
$c \mapsto \ell(c)^{\langle 1 \rangle} \otimes
 \ell(c)^{\langle 2 \rangle}$ as follows:
\begin{align}                                                       
\label{el1}
&\ell(c)^{\langle 1 \rangle}
\psi(e\otimes {\ell(c)^{\langle 2 \rangle}})
=\ell(c)^{\langle 1 \rangle}{\ell(c)^{\langle 2 \rangle}}_{(0)}\otimes
{\ell(c)^{\langle 2 \rangle}}_{(1)}=1\otimes c,\\
\label{elleft}
&\ell(c)^{\langle 1 \rangle}\otimes \psi(e\otimes{\ell(c)^{\langle 2 \rangle}})
= \ell(c)^{\langle 1 \rangle}\otimes {\ell(c)^{\langle 2 \rangle}}_{(0)}\otimes
{\ell(c)^{\langle 2 \rangle}}_{(1)}
=\ell(c_{(1)})^{\langle 1 \rangle}\otimes \ell(c_{(1)})^{\langle 2 \rangle}
\otimes c_{(2)},\\                                                  \label{elright}
&\psi^{-1}({\ell(c)^{\langle 1 \rangle}}\otimes e)
\otimes \ell(c)^{\langle 2 \rangle}
= {\ell(c)^{\langle 1 \rangle}}_{(-1)} \otimes {\ell(c)^{\langle 1 \rangle}}_{(0)} 
\otimes \ell(c)^{\langle 2 \rangle}
=c_{(1)}\otimes\ell(c_{(2)})^{\langle 1 \rangle}\otimes
\ell(c_{(2)})^{\langle 2 \rangle}.                                 
\end{align}
Applying $\id\otimes\varepsilon$ to \eqref{el1} yields a useful
formula
\begin{equation}\label{elco}
\ell(c)^{\langle 1 \rangle}{\ell(c)^{\langle 2 \rangle}}=\varepsilon(c).
\end{equation}

It is worthwhile to observe the left-right symmetry of principal extensions. 
We already noted  (see \eqref{lr}) the equality of the left and right coaction-invariant 
subalgebras. Now let us define the left canonical map as 
\[
\can_L : P\underset{B}{\ot} P \ni p\ot q \longmapsto p_{(-1)} \ot p_{(0)}q\in C\ot P. 
\]
One can check that it 
is related to the right canonical map $\can$ by the formula \cite{bhems} 
\[
   \psi\circ \can_L = \can.
\]
Also, if $\ell$ is a strong connection and $\tcan_L := (\id\ot m)\circ ({}_P\Delta\ot\id)$ is 
the lifted left canonical map, then $\tcan_L\circ\ell=\id\ot 1$. Hence 
\[  \label{C1}
c\ot p \longmapsto \ell(c)^{\ip{1}}\ot \ell(c)^{\ip{2}}p  
\]
is a splitting of $\tcan_L$ just as 
\[  \label{C2}
  p\ot c\longmapsto p\hs \ell(c)^{\ip{1}}\ot \ell(c)^{\ip{2}}
\]
is a splitting of $\tcan$. 

\begin{lem} \label{L1}
Let $P$ be an object in the category ${}_e^C\!\mathbf{Alg}_e^C$ of all unital algebras with 
$e$-coaugmented left and right $C$-coactions. Assume that there exists an invertible entwining $\psi: C\ot P\ra P\ot C$ making $P$ an entwined module. Then, if $P$ admits a strong connection $\ell$,  it is principal.
\end{lem}
\begin{proof}
Following \cite{bh04}, first we argue that 
\begin{equation}\label{ei1}
	\s : P\ni p \longmapsto p_{(0)}\ell(p_{(1)})^{\langle 1 \rangle}
	\otimes \ell(p_{(1)})^{\langle 2 \rangle} \in
	B \otimes P
\end{equation}
is a left $B$-linear
splitting of the multiplication map.
Indeed, $m\circ \s=\id$ because of \eqref{elco}, and the calculation 
\[
\psi( e\ot p_{(0)}\ell(p_{(1)})^{\langle 1 \rangle})\ot \ell(p_{(1)})^{\langle 2 \rangle}
=p_{(0)}\ell(p_{(1)})^{\langle 1 \rangle} \ot e \ot \ell(p_{(1)})^{\langle 2 \rangle}
\]
obtained  using \eqref{entwining} proves that $\s(P)\subseteq B\ot P$. 
This splitting is evidently right $C$-colinear, so that its existence proves the equivariant projectivity.  

Next, let us check that the formula
\begin{equation}\label{ei0}
\can^{-1} : P \otimes C \longrightarrow P \underset{B}{\otimes} P,\quad
	p \otimes c \longmapsto p\ell(c)^{\langle 1 \rangle} 
\underset{B}{\otimes} \ell(c)^{\langle 2 \rangle},
\end{equation}
defines the inverse of the canonical map $\can$, so that the
coaction of $C$ is Galois.
It follows from \eqref{el1} that 
\[
\can(\can^{-1} (p\ot c)) 
=   p\hs \ell(c)^{\hsp\langle 1 \rangle} {\ell(c)^{\hsp\langle 2 \rangle}}_{\!(0)} 
\ot {\ell(c)^{\hsp\langle 2 \rangle}}_{\!(1)}
=p\ot c\,.
\]
On the other hand, taking advantage of \eqref{elco} and \eqref{ei1},  we see that 
\[
\can^{-1} (\can (p \underset{B}{\ot} q)) 
= p q_{(0)}\ell(q_{(1)} )^{\hsp\langle 1 \rangle}\underset{B}{\ot}\ell(q_{(1)} )^{\hsp\langle 2 \rangle}
= p\underset{B}{\ot} q_{(0)}\ell(q_{(1)} )^{\hsp\langle 1 \rangle}\ell(q_{(1)} )^{\hsp\langle 2 \rangle}
=p\underset{B}{\ot} q\, .
\]
Thus  the conditions (i) and (iv) of the principality of a $C$-extension are satisfied. Finally, Condition~(ii) is simply assumed, and 
Condition~(iii) follows from the uniqueness of an entwining that makes $P$ an entwined module. 
\end{proof}

Note that, if there exists a strong connection $\ell$, then  \eqref{ei0} yields 
\[                                           \label{tell}
 \tau(c) =  \ell(c)^{\langle 1 \rangle} 
\underset{B}{\otimes} \ell(c)^{\langle 2 \rangle}. 
\]
In the Heyneman-Sweedler type notation, we write $\tau(c) = \tau(c)^{[1]} \ot_B \tau(c)^{[2]}$. 
Then the canonical entwining reads 
\[                                         \label{psil}
 \psi(c\ot p)= \tau(c)^{[1]} (\tau(c)^{[2]}\hs p)_{(0)} \ot (\tau(c)^{[2]}\hs p)_{(1)}
=\ell(c)^{\ip{1}} (\ell(c)^{\ip{2}}\hs p)_{(0)} \ot (\ell(c)^{\ip{2}}\hs p)_{(1)}. 
\]

\begin{rem}\em
In \cite{bh04}, there is the converse statement: if $P$ is principal, it admits a strong
connection. Thus
principal extensions can be characterized as these that admit a strong
connection.
\end{rem}

Recall now that classical principal bundles can be viewed as functors
transforming finite-dimensional vector spaces into associated
vector bundles.
Analogously, one can prove that a principal $C$-extension $B\inc P$
defines a functor from the category of finite-dimensional left
$C$-comodules into the category of finitely generated projective
left $B$-modules~\cite{bh04}. Explicitly, if $V$ is
 a left $C$-comodule  with coaction ${}_V\Delta$, this functor
assigns to it the cotensor product
\begin{equation}
 P\,\underset{C}{\Box}\,V:=\{ \mbox{\small$\sum_i$}\; p_i\otimes v_i \in 
P\otimes V\;|\;
 \mbox{\small$\sum_i$}\;
 \Delta_P(p_i)\otimes v_i= \mbox{\small$\sum_i$}\; p_i\otimes {}_V\Delta(v_i)\}.
\end{equation}
In particular, if $g\in C$ is a group-like element,
the formula ${}_{\C}\Delta(1):=g\hs\otimes\hs 1$ defines
a 1-dimen\-sional corepresentation, and
\[\label{line}
 P\,\hs\underset{C}{\Box}\,\hs\C\hsp=\hsp\{ p\hsp\in\hsp P \;|\; 
\Delta_P(p)\!=\!p\otimes g\} =: P_g
\]
can be viewed as a noncommutative associated complex line bundle.
Then
the general formula for computing an idempotent $E_g$
of the associated module $P_g$
out of a corepresentation and a strong
connection becomes a very simple special case of 
\cite[Theorem~3.1]{bh04}: 
\[\label{ie}
P_g\cong B^n E_g\,,\quad \left(E_g\right)_{i,j=1}^n:=\left(g^R_i \hs g^L_j\right)_{i,j=1}^n,
\quad \ell(g)=:\sum_{k=1}^n g^L_k\otimes g^R_k
\in P_{g^{-1}}\otimes P_g\, .
\]

A fundamental special case of principal extensions is provided
by {\em principal comodule algebras}. One assumes
then that $C=H$ is a Hopf algebra with  bijective antipode~$\hk$,
the canonical map is bijective, and $P$ is an $H$-equivariantly
projective left $B$-module. This brings us in touch with 
compact quantum groups. Assume that $\bar H$ is the 
$C^*$-algebra of a compact quantum group in the sense
of Woronowicz \cite{w-sl87,w-sl98}, and
$H$ is its dense Hopf *-subalgebra
spanned by the matrix coefficients of the irreducible
unitary corepresentations. Let $\bar P$ be a unital $C^*$-algebra
and $\delta:\bar P\rightarrow  \bar P\barot \bar H$
an injective $C^*$-algebraic right coaction of $\bar H$ on $\bar P$. 
(See \cite[Definition 0.2]{bs93} for a general definition and
  \cite[Definition 1]{b-fp95} for the special case of compact quantum 
groups.)
Here $\barot$ denotes the minimal $C^*$-completion of
the algebraic tensor product $\bar P\otimes \bar H$.

To extend Woronowicz's Peter-Weyl theory \cite{w-sl98} from compact
quantum groups to compact quantum principal bundles, 
one  defines \cite{bh}
the subalgebra $\PWG(\bar P)\subseteq \bar P$ of elements for
which the coaction lands in $\bar P \otimes H$, i.e.\
\begin{equation}
\PWG(\bar P):=\{p\in \bar P\;|\;\delta(p)\in P\otimes H\}.
\end{equation}
One easily checks that it is an $H$-comodule algebra. 
We call $\PWG(\bar P)$ the {\em Peter-Weyl comodule algebra}
 associated to the $C^*$-coaction~$\hd$.
It follows from results of
\cite{b-fp95} and \cite{p-p95} that $\PWG(\bar P)$ is dense in $\bar P$.
Also, it is straightforward to verify  \cite{bh} that the
operation $\bar P \mapsto \PWG(\bar P)$ gives
 a functor commuting with
taking fibre products (pullbacks), and 
that $\PWG(\bar P)^{\co H}$ coincides with the $C^*$-algebra 
$\bar P^{\co \bar H}$.

Finally, let us remark that, for  a compact Hausdorff topological
group $G$ and a unital $C^*$-algebra $A$, 
we can use the isomorphism $A\,\bar \ot \,\CC(G)\cong \CC(G,A)$ 
(e.g.~see \cite[Corollary~T.6.17]{w-ne93})
to translate a right  $\CC(G)$-coaction $\delta$
into a  $G$-action 
$\chi\colon G\ni g\mapsto \chi_g\in \mathrm{Aut}(A)$ as follows:
\[\label{aut}
\delta\colon A\longrightarrow A\,\bar{\otimes}\, C(G)\cong C(G,A),
\quad \delta(a)(g)=:\chi_g(a)\,.
\]
Thus we can use the terminology of
right $\CC(G)$-comodule $C^*$-algebras and 
$G$-$C^*$-algebras synonymously. It is important to bear in mind that
the Peter-Weyl functor maps
$G$-equivariant *-homomorphisms to colinear homomorphisms of 
right $\cO(G)$-comodule algebras~\cite{bh}.

\subsection{Standard Hopf fibration of quantum SU(2)}            
\label{sPs}

The standard quantum Hopf fibration is given by an action of
$\mathrm{U}(1)$ on the quantum group $\mathrm{SU}_q(2)$,
$q\in(0,1)$.
The
 coordinate ring of $\SUq$
is generated by
$\ha$, $\hb$, $\hc$, $\hd$ with relations
\begin{align}
& \ha \hb =q \hb \ha,\quad \ha \hc =q\hc \ha,\quad \hb \hd = q \hd \hb,\quad
\hc \hd = q \hd \hc, \quad \hb \hc =\hc \hb,\\
& \ha \hd - q \hb \hc = 1, \quad  \hd \ha - q^{-1} \hb \hc = 1,
\end{align}
and involution $\ha^*=\hd$, $\hb^*=-q\hc$.
It is a Hopf *-algebra with
comultiplication $\Delta$, counit $\varepsilon$, and antipode
$\hk$ given by
\begin{align}
& \Delta(\ha) = \ha \otimes \ha + \hb \otimes \hc, \quad
\Delta(\hb) = \ha \otimes \hb + \hb \otimes \hd,\\
& \Delta(\hc) = \hc \otimes \ha + \hd \otimes \hc, \quad
\Delta(\hd) = \hc \otimes \hb + \hd \otimes \hd,   \\
& \varepsilon(\ha)=\varepsilon(\hd)=1, \quad
\varepsilon(\hb)=\varepsilon(\hc)=0,\\
& \hk (\ha)=\hd,\quad \hk (\hb)=-q^{-1}\hb,
\quad \hk (\hc)=-q \hc,\quad
\hk (\hd)=\ha.
\end{align}

Let $\OU$ denote the commutative and cocommutative Peter-Weyl Hopf *-algebra 
of $\mathrm{U}(1)$, and let $\V$ stand for its unitary group-like generator. 
Note that the counit $\vare$  and the antipode $\hk$
satisfy $\vare({\V})=1$ and $\hk({\V})={\V}^{*}$. 
There is a Hopf *-algebra surjection
\[    \label{Hsur}
\pi:\SUq\lra \OU, \quad \pi(\ha)={\V},\ \ \pi(\hd)={\V}^{-1},\ \ 
\pi(\hb)=\pi(\hc)=0.
\]
Setting $\Delta_R:=(\id\otimes\pi)\circ \Delta$, we
obtain a right \OU-coaction on \SUq.
On generators, the coaction reads
\[
 \Delta_R(\ha)\hsp =\hsp\ha\otimes {\V}, \quad\!
\Delta_R(\hb) \hsp=\hsp  \hb\otimes {\V}^{-1},\quad\!
 \Delta_R(\hc) \hsp=\hsp  \hc\otimes {\V}, \quad\!
\Delta_R(\hd) \hsp=\hsp \hd \otimes {\V}^{-1}.
\]
The *-subalgebra of coaction invariants
defines the coordinate ring of the
standard Podle\'s quantum sphere:
\begin{equation}\label{SqInSUq}
\Sq:=\SUq^{{\co \cO(\mathrm{U}(1))}}
=\left\lbrace a\in\SUq \;|\; \Delta_R(a)=a \otimes 1\right\rbrace.
\end{equation}

One can prove (see~\cite{p-p87}) that $\Sq$ is
 isomorphic to the *-algebra
generated by $B$ and  hermitean  $A$
satisfying the relations 
\[\label{podles}
AB=q^{2}BA,\quad B^{*}B=A-A^{2},\quad BB^{*}=q^{2}A-q^{4}A^{2}.
\]
An isomorphism  is explicitly given by the formulas
$A=-q^{-1}\hb\hc$ and $B=-\hb\ha$.
The irreducible Hilbert space representations of \Sq\ are given by
\begin{align}
&\rho_0(A)=\rho_0(B)=0,\quad \rho_0(1)=1\quad \mbox{on}\ \, \H=\C,\\
&\rho_+(A)e_n= q^{2n}e_n,\quad \rho_+(B)e_n=q^{n}(1-q^{2n})^{1/2}e_{n-1}
\quad \mbox{on}\ \, \H=\lN,
\end{align}
where $\left\lbrace e_n \;|\; n=0,1,\ldots\right\rbrace$ is an
orthonormal basis of $\lN$.

Recall that the universal $C^*$-algebra
of a complex *-algebra is
the $C^*$-completion with respect to the
universal $C^*$-norm given by the
supremum (if it exists) of the operator norms over all bounded
*-representations.
Let $\CSq$ denote the universal $C^*$-algebra generated by
$A$ and $B$ satisfying~\eqref{podles}.
From the above representations, it follows
that $\CSq$ is the minimal unitalization of $\cK(\lN)$, that is, 
\begin{gather}\label{CSq}
\CSq\,\cong\, \cK(\lN) \oplus\C\,\subseteq\, \B(\lN), \\ (k+\ha)(k'+\ha')=(kk'+ \ha' k + \ha k')+ \ha\ha', 
\quad  k,k'\in \cK(\lN),\ \, \ha,\ha'\in\C. \label{g}
\end{gather}
Here $\cK(\lN)$ and $\B(\lN)$ denote the 
$C^*$-algebras of compact and bounded
operators  respectively on the Hilbert space~$\lN$.
The isomorphism \eqref{CSq} implies that ${\K}_0(\CSq)\cong\Z\oplus\Z$,
where one generator of $K$-theory is given by the class of the unit 
$1\in\CSq$,
and the other by the class of the 1-dimensional projection
onto $\C e_0\subseteq \lN$.

Furthermore, ${\K}^{0}(\CSq)\cong\Z\oplus\Z$. 
We identify one generator of $K$-homology with the class of
the pair of representations $[(\id,\vare)]$,
where $\id(k+\alpha)=k+\alpha$ and
$\vare(k+\alpha)=\alpha$ for all $k\in \cK(\lN)$ and $\ha \in\C$.
The other generator can be given by the class
of the pair of representations $[(\vare,\vare_0)]$
with the (non-unital) representation $\vare_0$ of $\cK(\lN) \oplus\C$
defined by $\vare_0(k+\alpha)= \alpha {\S}{\S}^{*}$, where
\begin{equation}\label{S}
{\S}:\lN\lra\lN, \quad {\S}e_n=e_{n+1},
\end{equation}
denotes the unilateral shift on $\lN$.
(See \cite{MNW1} for a detailed treatment of the $K$-homology
and $K$-theory of Podle\'s quantum spheres.)

We shall also consider the coordinate ring of the quantum disc $\Dq$
generated by $z$ and $z^{*}$ with the relation
\begin{equation}
z^* z -q^2 z z^* = 1- q^2.
\end{equation}
Its bounded irreducible Hilbert space representations
are given by
\begin{align}
&\mu_\theta(z)={\mathrm e}^{\im\theta}\quad \mbox{on}\ \, \H=\C,
\quad \theta\in[0,2\pi),\\
&\mu(z)e_n= (1-q^{2(n+1)})^{1/2}e_{n+1}
\quad \mbox{on}\ \, \H=\lN. \label{mu}
\end{align}
It has been shown in \cite{kl93}
that the universal $C^*$-algebra of $\Dq$ is isomorphic to
the Toeplitz algebra given as the $C^*$-algebra
generated by the unilateral shift ${\S}$ of Equation~\eqref{S}.
The representation $\mu$ defines then an embedding
of $\Dq$ into~$\T$.

Let $\CS$ denote the $C^*$-algebra of continuous 
functions on the unit circle $\mathrm{S}^1$,
and let ${\U}$ be its unitary generator. 
The Toeplitz algebra gives rise to
the following short exact sequence of $C^*$-algebras:
\begin{equation}\label{sesToeplitz}
 0\,\lra\, \cK(\lN)\,\lra\, \T\, 
{\overset{\sigma}{\lra}}\,\, \CS\,\lra\, 0.
\end{equation}
Here the so-called symbol map $\sigma :\T\ra\CS$ is given by
$\sigma({\S})={\U}$. Since ${\S}-\mu(z)$ belongs to $\cK(\lN)$, 
it follows in particular
that $\sigma(\mu(z))={\U}$.

Now let us consider the associated quantum line bundles
as finitely generated projective modules.
They are defined by the 1-dimen\-sional
corepresentations  
$\C\ni 1\mapsto {\V}^{N}\otimes 1$, $N\in\Z$, as cotensor
products~\eqref{line}:
\begin{equation} \label{MsubN}
M_N:= \{ p\in\SUq \;|\; \Delta_R(p)=p\otimes {\V}^{N}\}.
\end{equation}
Since $\Delta_R$ is a morphism of algebras, $M_N$ is
an \Sq-bimodule. Our next step is to determine explicitly
projections describing these projective modules.

For $l\in\frac{1}{2}{\N}$ and $i,j=-l,-l+1,\ldots, l$,
let $t^{l}_{i,j}$ denote the matrix elements of the irreducible unitary
corepresentations of \SUq, that is, 
\begin{equation}\label{matcoef}
\Delta(t^{l}_{i,j})=\sum_{k=-l}^{l} t^{l}_{i,k}\otimes t^{l}_{k,j}\,, \qquad
\sum_{k=-l}^{l} t^{l*}_{k,i}\hs t^{l}_{k,j}=\sum_{k=-l}^{l}
 t^{l}_{i,k}\hs t^{l*}_{j,k}=\delta_{ij}\,.
\end{equation}
By the Peter-Weyl theorem for compact quantum groups~\cite{W},
$\SUq=\oplus_{l\in\frac{1}{2}{\N}}\oplus_{i,j=-l}^{l}\C\hs t^{l}_{i,j}$.
From the explicit description of $t^{l}_{i,j}$ \cite[Section~4.2.4]{KS} and the definition
of $\Delta_R$, it follows that 
$\Delta_R(t^{l}_{i,j})=t^{l}_{i,j} \otimes {\V}^{-2j}$,
so that $t^{l}_{i,-j}\in M_{2j}$.
It can be shown \cite{HM,SW} that
$t^{|j|}_{i,-j}$, $i=-|j|,\ldots, |j|$ generate $M_{2j}$ as a left
 \Sq-module and
$M_{2j}\cong \Sq^{2|j|+1} E_{2j}$ for all $j\in\frac{1}{2}\Z$, where
\begin{equation}                                                       \label{EN}
 E_{2j} =
  \begin{pmatrix}
   t^{|j|}_{-|j|,-j} \\
   \vdots \\
t^{|j|}_{|j|,-j}
  \end{pmatrix}
\begin{pmatrix}
   t^{|j|*}_{-|j|,-j}\,, &\!\!
   \cdots\,, &\!\!
t^{|j|*}_{|j|,-j}
  \end{pmatrix}
\in\Mat_{2|j|+1}(\Sq).
\end{equation}
It is clear that $E_{2j}^{2}=E_{2j}$ due to \eqref{matcoef}
and $E_{2j}^{*}=E_{2j}$, so that $E_{2j}$ is a projection.

\section{Principality of one-surjective pullbacks}\label{fppe}

We begin by defining an ambient category for
pullback diagrams appearing in the second part of this section. 
Let $P$ be a unital algebra equipped with both a right 
coaction $\Delta_P:P\ra P\ot C$ and a left coaction
${}_P\Delta:P\ra C\ot P$ of the same coalgebra~$C$.
We do {\em not} assume that these coactions commute, but we
do assume that they are coaugmented by the same group-like
element $e\in C$, i.e., $\Delta_P(1)=1\ot e$ and ${}_P\Delta(1)=e\ot 1$.
For a fixed  coalgebra $C$ and a group-like $e\in C$, we consider
the category ${}_e^C\!\mathbf{Alg}_e^C$ of all such unital algebras with 
$e$-coaugmented left and right $C$-coactions. 
Here morphisms are bicolinear
algebra homomorphisms.

Since we work
over a field, this category is evidently closed under any pullbacks.
If 
$\pi_1:~P_1~\rightarrow~P_{12}$
and  $\pi_2:P_2\rightarrow P_{12}$ are morphisms in 
${}_e^C\!\mathbf{Alg}_e^C$,
then the fibre product algebra $P:=P_1{\times}_{(\pi_1,\pi_2)} P_2$
becomes a right $C$-comodule via
\begin{equation}
\Delta_P(p,q) = ({p}_{(0)},0)\otimes {p}_{(1)}
+(0,{q}_{(0)})\otimes {q}_{(1)},
\end{equation}
and a left $C$-comodule via
\begin{equation}
{}_P \Delta(p,q) = {p}_{(-1)}\otimes ({p}_{(0)},0)
+{q}_{(-1)}\otimes(0,{q}_{(0)}).
\end{equation}
Also, it is clear that $\Delta_P(1,1) = (1,1)\otimes e$
and ${}_P \Delta(1,1) = e\otimes(1,1)$.

\subsection{Principality of images and preimages}

In the following lemma, we prove that any surjective morphism in 
${}_e^C\!\mathbf{Alg}_e^C$
whose domain is a principal extension can be split by a left colinear 
map and by a right colinear map (not necessarily by a bicolinear map). 
Note that 
the first part of the lemma is proved much the same way as in the
Hopf-Galois case \cite[Lemma~3.1]{hkmz}: 

\begin{lem}\label{fer}
Let 
$ \pi : P \rightarrow Q$ be a surjective
morphism in
the category ${}_e^C\!\mathbf{Alg}_e^C$ of  unital algebras with 
$e$-coaugmented left and right $C$-coactions.
If $P$ is principal, then:
\vspace*{-1mm}\begin{enumerate}
\item[(i)] The induced map
$\pi^\coc : P^\coc \rightarrow Q^\coc$ is surjective.
\item[(ii)] There exists a unital right $C$-colinear splitting of~$\pi$.
\item[(iii)] There exists a unital left $C$-colinear splitting of~$\pi$. 
\item[(iv)] $Q$ is principal.
\end{enumerate}
Furthermore, if $Q^{\prime}\in {}_e^C\!\mathbf{Alg}_e^C$, \,$Q^{\prime}\subseteq Q$, 
is principal, then so is $\pi^{-1}(Q^{\prime})$. 
\end{lem}
\begin{proof}
It follows from the right colinearity and surjectivity of $\pi$ 
that $ \pi (P^{\co C}) \subseteq Q^{\co C}$. To prove the
converse inclusion,  we take advantage of the left $P^{\co C}$-linear
retraction of the inclusion $P^{\co C}\subseteq P$ that was used to
prove \cite[Theorem~2.5(3)]{bh04}:
\[
\sigma_\varphi:P\longrightarrow P^{\co C},\quad 
\sigma_\varphi(p):=p_{(0)}\ell(p_{(1)})^{\langle 1 \rangle}
\varphi(\ell(p_{(1)})^{\langle 2 \rangle})\,.
\]
Here $\ell$ is a strong connection on $P$ and $\varphi$ is any unital
linear functional on $P$. 
It follows from \eqref{ei1} that $\sigma_\varphi(p)\in P^{\co C}$. 
If $\pi(p)\in Q^{\co C}$, then 
$\sigma_\varphi(p)$ is a desired element of $P^{\co C}$ that is mapped by $\pi$  
to $\pi(p)$. Indeed, since $\pi(p_{(0)})\otimes p_{(1)}=\pi(p)_{(0)}\otimes \pi(p)_{(1)}=\pi(p)\otimes e$, 
using the unitality of $\pi$, $\varphi$, and $\ell(e)=1\ot 1$,
  we compute
\[
\pi(\sigma_\varphi(p))=\pi(p_{(0)})\pi(\ell(p_{(1)})^{\langle 1 \rangle})
\varphi(\ell(p_{(1)})^{\langle 2 \rangle})=\pi(p).
\]

To show the second assertion, let us choose 
any unital $k$-linear splitting of $\pi\!\!\upharpoonright_{\hsp P^{\co C}}$ and denote it 
by~$\alpha^{\co C}$.
We want to prove that the formula 
\[
\alpha_R(q):= 
\alpha^{\co C} (q_{(0)}\pi (\ell(q_{(1)})^{\langle 1 \rangle})) 
\ell(q_{(1)})^{\langle 2 \rangle}
\]
 defines a unital right colinear splitting of~$\pi$. 
Since $\pi$ is surjective, we can write $q=\pi(p)$. Then, using properties of $\pi$, 
we obtain: 
\begin{align}
q_{(0)}\hs \pi(\ell(q_{(1)})^{\ip{1}})\!\ot \ell(q_{(1)})^{\ip{2}}
&= \pi(p)_{(0)}\hs \pi(\ell(\pi(p)_{(1)})^{\ip{1}})\!\ot \ell(\pi(p)_{(1)})^{\ip{2}} \nonumber\\
&= \pi(p_{(0)})\hs \pi(\ell(p_{(1)})^{\ip{1}})\!\ot \ell(p_{(1)})^{\ip{2}} \nonumber\\
&= \pi(p_{(0)}\ell(p_{(1)})^{\ip{1}})\!\ot \ell(p_{(1)})^{\ip{2}}. 
\end{align}
Now it follows from \eqref{ei1} that the above tensor is in $Q^{\co C}\ot P$. 
Hence $\alpha_R$ is well defined. It is straightforward to verify that 
$\alpha_R$ is unital, right colinear, and splits $\pi$. (Note that, since $q\in Q^{\co C}$ 
implies $q_{(0)}\ot q_{(1)} = q\ot e$, we have 
$\alpha^{\co C} =\alpha_R\!\upharpoonright_{\hsp Q^{\co C}}$.)
The third assertion can be proven in an analogous manner. 

To prove $(iv)$, we first show that 
the inverse of the canonical map 
$\can_Q: Q\ot_{Q^\coc} Q{\ra} Q\ot C$ (see \eqref{can}) 
is given by 
\[
 \can_Q^{-1} : Q \ot C\lra Q\underset{Q^\coc}{\ot} Q, \quad 
q\ot c\longmapsto q\pi(\ell(c)^{\langle 1 \rangle})\underset{Q^\coc}{\ot}\pi(\ell(c)^{\langle 2 \rangle}). 
\]
Using the properties of $\pi$ and  $\ell$, we get 
\begin{align}                     \nonumber
 (\can_Q\circ \can_Q^{-1}) \,(\pi(p)\ot c) 
&= \can_Q\Big( \pi(p\ell(c)^{\langle 1 \rangle})\underset{Q^\coc}{\ot}\pi(\ell(c)^{\langle 2 \rangle})\Big)  \\
&= \pi\left(p\,\ell(c)^{\langle 1 \rangle}\,
{\ell(c)^{\langle 2 \rangle}}_{\!(0)}\right)\ot {\ell(c)^{\langle 2 \rangle}}_{\!(1)}\nonumber \\
&=\pi(p)\ot c\hs .
\label{R}
\end{align}
Similarly, 
\begin{align}                                             \nonumber
 (\can_Q^{-1} \circ\can_Q)\Big(\pi(p)\underset{Q^\coc}{\ot} \pi(p^{\hs\prime})\Big)
 &=\can_Q^{-1}\Big(\pi(pp^{\hs\prime}_{(0)})\ot p^{\hs\prime}_{(1)}\Big)
\nonumber \\
 &= \pi(pp^{\hs\prime}_{(0)}\ell(p^{\hs\prime}_{(1)} )^{\langle 1 \rangle})
\underset{Q^\coc}{\ot} \pi( \ell(p^{\hs\prime}_{(1)})^{\langle 2 \rangle})\\
&= \pi(p)\underset{Q^\coc}{\ot} \pi(p^{\hs\prime}_{(0)}\ell(p^{\hs\prime}_{(1)} )^{\langle 1 \rangle} 
\ell(p^{\hs\prime}_{(1)})^{\langle 2 \rangle})\nonumber \\
&= \pi(p)\underset{Q^\coc}{\ot}   \pi(p^{\hs\prime}). 
\label{tcan}         
\end{align}
Here we used the fact that 
$\pi(p^{\hs\prime}_{(0)}\ell(p^{\hs\prime}_{(1)} )^{\langle 1 \rangle})
\ot \ell(p^{\hs\prime}_{(1)})^{\langle 2 \rangle}\in Q^\coc\ot P$. 
Hence the extension $Q^\coc\subseteq Q$ is Galois, and we have 
the canonical entwining $\psi_Q: C\ot Q\ra Q\ot C$.

Our next aim is to prove that $\psi_Q$ is bijective. 
We know by assumption that the canonical entwining 
$\psi_P : C\ot P\ra P\ot C$ is invertible. 
To determine its inverse, recall that the left and right coactions are 
given by $\psi_P^{-1}(p\ot e)$ and $\psi_P(e\ot p)$, respectively. 
Then apply \eqref{entwining} to compute 
 \begin{align} \nonumber
 \psi_P\Big( (p\hs \ell(c)^{\!\ip{1}})_{\hsp(-1)} \ot (p\hs \ell(c)^{\!\ip{1}})_{\hsp(0)}\hs\ell(c)^{\!\ip{2}}\Big) 
 &= p\hs \ell(c)^{\!\ip{1}}\psi_P\Big( e\ot \ell(c)^{\!\ip{2}}\Big) \\
&=p\hs {\ell(c)^{\!\ip{1}}\ell(c)^{\!\ip{2}}}_{\!(0)}\ot {\ell(c)^{\!\ip{2}}}_{\!(1)}\nonumber \\
&=\ p\ot c.  \label{2.11}
 \end{align}
Hence $\psi_P^{-1}(p\ot c)
=(p\hs \ell(c)^{\hsp\ip{1}})_{\hsp(-1)}\ot (p\hs \ell(c)^{\hsp\ip{1}})_{\hsp(0)}\hs \ell(c)^{\hsp\ip{2}}$. 
On the other hand, 
\begin{align} \nonumber
 \psi_Q(c\ot \pi(p))&=\pi(\ell(c)^{\hsp\ip{1}}) \big(\pi(\ell(c)^{\hsp\ip{2}})\hs \pi(p)\big)_{\!(0)} 
\ot \big(\pi(\ell(c)^{\hsp\ip{2}})\hs \pi(p)\big)_{\!(1)}\\
&=\pi\big(\ell(c)^{\hsp\ip{1}}\big)\, \pi\big((\ell(c)^{\hsp\ip{2}}\hs p)_{\hsp(0)} \big)
\ot (\ell(c)^{\hsp\ip{2}}\hs p)_{\hsp(1)}\nonumber \\
&= (\pi\ot\id)\big(\psi_P(c\ot p)\big) ,    \label{pipsi}  \\[8pt]
\nonumber
(\id\ot \pi)\big( \psi_P^{-1}(p\ot c) \big)  
&=(p\,\ell(c)^{\hsp\ip{1}})_{\hsp(-1)}  \ot 
\pi\big((p\,\ell(c)^{\hsp\ip{1}})_{\hsp(0)}\big)\, \pi\big(\ell(c)^{\hsp\ip{2}}\big)\\
&= \big(\pi(p\,\ell(c)^{\hsp\ip{1}})\big)_{\!(-1)}  \ot 
\big(\pi(p\,\ell(c)^{\hsp\ip{1}})\big)_{\!(0)}\, \pi(\ell(c)^{\hsp\ip{2}})\nonumber \\
&= {}_Q\Delta\llp\pi(p)\hs \pi(\ell(c)^{\hsp\ip{1}})\lrp\hs \pi(\ell(c)^{\hsp\ip{2}}).
 \label{pipsi-1}
\end{align}
The second part of the above computation implies that the assignment 
\[
\psi_Q^{-1}:Q\ot C\lra C\ot Q, \quad \pi(p)\ot c\longmapsto (\id\ot \pi)
(\psi_P^{-1}(p\ot c))         \label{psiinv}
\]
is well defined. 
Now it follows from the first part that $\psi_Q^{-1}$ 
is the inverse of $\psi_Q$: 
\[\label{psiphi}
\psi_Q\Big(\psi_Q^{-1}\big(\pi(p)\ot c\big)\Big)
=\psi_Q\Big( (\id\ot \pi)\big(  \psi_P^{-1}(p\ot c)\big)\Big)
= (\pi\ot\id )\Big(\psi_P\big(  \psi_P^{-1}(p\ot c)\big)\Big)
=\pi(p)\ot c,
\]\[
                           \label{phipsi}
 \psi_Q^{-1}\Big(\psi_Q\big(c\ot\pi(p) \big)\Big)
= \psi_Q^{-1}\Big( (\pi\ot \id)\big(  \psi_P(c\ot p)\big)\Big)
= (\id\ot\pi )\Big(\psi_P^{-1}\big(  \psi_P(c\ot p)\big)\Big)
= c\ot\pi(p). 
\]

On the other hand, we observe that $(\pi\ot\pi)\circ \ell$ is a strong connection on $Q$. Combined 
 with the just proven existence of a bijective entwining that makes $Q$ 
an entwined module, it allows us to apply Lemma~\ref{L1} and conclude the proof of $(iv)$. 

To prove the final statement of the lemma, note first that  
$\pi^{-1}(Q^{\prime})\in {}_e^C\!\mathbf{Alg}_e^C$.
Next, observe that,  
 if $\ell^{\prime} : C \ra Q^\prime\ot Q^\prime$ is a strong connection on $Q^\prime$, 
then it is also a strong connection on $Q$. 
Now, it follows from \eqref{psil} that for any $q\in Q^\prime$ 
\[
 \psi_Q(c\ot q) 
=\ell^\prime(c)^{\ip{1}} \hsp
 \left(\ell^\prime(c)^{\ip{2}}\hs q \right)_{\hsp(0)} 
 \ot \left(\ell^\prime(c)^{\ip{2}}\hs q\right)_{\hsp(1)}\in Q^\prime \ot C. 
\]
Much the same way, it follows from the $Q$-analog of the formula
following \eqref{2.11} that\linebreak
$\psi_Q^{-1} (Q^\prime\ot C)\subseteq C\ot Q^\prime$. Hence to see that 
$\psi_P$ and $\psi_P^{-1}$ restrict  to $\pi^{-1}(Q^{\prime})$, 
we can apply \eqref{pipsi} and \eqref{psiinv}, respectively. 

A key step now is to construct  a strong connection on $\pi^{-1}(Q^\prime)$. 
Let $\ha_R$ and $\ha_L$ be, respectively, right and  left colinear unital 
splittings of $\pi$. Their existence is guaranteed by the already proven $(ii)$ and $(iii)$. 
The map $(\ha_L\ot\ha_R)\circ\ell^\prime : C \ra \pi^{-1}(Q)\ot \pi^{-1}(Q)$ is 
bicolinear and satisfies 
\[
\ha_L(\ell^\prime(e)^{\ip{1}})\ot \ha_R(\ell^\prime(e)^{\ip{2}}) =1\ot 1.
\]
However, 
\[
1\ot c-\big(\tcan \circ (\ha_L\ot\ha_R)\circ \ell^\prime  \big)(c) 
= 1\ot c - 
\ha_L(\ell^\prime(c_{(1)})^{\hsp\ip{1}})\hs \ha_R(\ell^\prime(c_{(1)})^{\hsp\ip{2}}) \ot c_{(2)} \neq 0.
\]
To solve this problem, 
we apply to it the splitting of the lifted canonical map given by a strong connection $\ell$ 
(see \eqref{C2}), and add to $(\ha_L\ot\ha_R)\circ\ell^\prime$\,:
\[ \label{lR}
\ell_R(c):= (\ha_L\ot\ha_R)(\ell^\prime(c)) + \ell(c) 
-  \ha_L(\ell^\prime(c_{(1)})^{\hsp\ip{1}})\hs \ha_R(\ell^\prime(c_{(1)})^{\hsp\ip{2}}) \hs 
\ell(c_{(2)}^{\ip{1}}) \ot \ell(c_{(2)}^{\ip{2}}). 
\]
Now $\tcan\circ \ell_R = 1\ot \id$, as needed. Also, $\ell_R(e)=1\ot 1$ 
and $((\pi\ot\id)\circ\ell_R)(C)\subseteq Q^\prime \ot P$. 
The right colinearity of $\ell_R$ is clear.
To check the left colinearity of $\ell_R$, 
using the fact that $P$ is a $\psi_P$ entwined and $e$-coaugmented module, 
we show that $(m_P\circ(\ha_L\ot \ha_R)\circ \ell^\prime )*\ell$ 
is left colinear. (Here $m_P$ is the multiplication of~$P$.) First we note that 
\[
({}_{P}\Delta \ot \id)\circ \llp(m_P\circ(\ha_L\ot \ha_R)\circ \ell^\prime )*\ell\lrp
= \llp\id \ot (m_P\circ(\ha_L\ot \ha_R)\circ \ell^\prime )*\ell\lrp\circ \Delta  \label{coli0}
\]
is equivalent to 
\begin{align}\nonumber
&\ha_L(\ell^\prime(c_{(1)})^{\ip{1}})\hs \ha_R(\ell^\prime(c_{(1)})^{\ip{2}})\hs \ell(c_{(2)})^{\ip{1}}\hsp \ot\hsp  e \hsp  \ot \hsp \ell(c_{(2)})^{\ip{2}} \\
&\hspace{40pt}= \psi_P\Big( c_{(1)} \ot  \ha_L(\ell^\prime(c_{(2)})^{\ip{1}})\hs \ha_R(\ell^\prime(c_{(2)})^{\ip{2}})
\hs \ell(c_{(3)})^{\ip{1}} \Big) \ot  \ell(c_{(3)})^{\ip{2}} . 
\end{align}
Since 
$c_{(1)} \ot  \ha_L(\ell^\prime(c_{(2)})^{\ip{1}}) \ot \ell^\prime(c_{(2)})^{\ip{2}}
= \psi_P^{-1}\big(\ha_L(\ell^\prime(c)^{\ip{1}})  \ot e \big)  \ot \ell^\prime(c)^{\ip{2}}$, we obtain 
\begin{align}
&\psi_P\Big(  c_{(1)} \ot  \ha_L(\ell^\prime(c_{(2)})^{\ip{1}})\hs \ha_R(\ell^\prime(c_{(2)})^{\ip{2}})
\hs \ell(c_{(3)})^{\ip{1}} \Big)  \ot  \ell(c_{(3)})^{\ip{2}} 
\nonumber\\
&\hspace{80pt}= \ha_L(\ell^\prime(c_{(1)})^{\ip{1}})\, 
\psi_P\Big(e \ot \ha_R(\ell^\prime(c_{(1)})^{\ip{2}})
\hs \ell(c_{(2)})^{\ip{1}} \Big)  \ot  \ell(c_{(2)})^{\ip{2}} 
\nonumber\\
&\hspace{80pt}= \ha_L(\ell^\prime(c_{(1)})^{\ip{1}})\hs \ha_R(\ell^\prime(c_{(1)})^{\ip{2}})\, 
\psi_P\Big(c_{(2)} \ot \ell(c_{(3)})^{\ip{1}}\Big) \ot \ell(c_{(3)})^{\ip{2}} 
\nonumber\\
&\hspace{80pt}= 
\ha_L(\ell^\prime(c_{(1)})^{\ip{1}})\hs \ha_R(\ell^\prime(c_{(1)})^{\ip{2}})\, 
\psi_P\Big(\psi_P^{-1} \big(\ell(c_{(2)})^{\ip{1}}\ot e\big)\Big)\ot  \ell(c_{(2)})^{\ip{2}} 
\nonumber\\
&\hspace{80pt}=\ha_L(\ell^\prime(c_{(1)})^{\ip{1}})\hs \ha_R(\ell^\prime(c_{(1)})^{\ip{2}})\, 
\ell(c_{(2)})^{\ip{1}}\ot e\ot  \ell(c_{(2)})^{\ip{2}} .\label{coli}
\end{align}

Hence $\ell_R$ is a strong connection with the property $\ell_R(C)\subseteq \pi^{-1}(Q^\prime)\ot P$. 
In a similar manner, we construct a strong connection $\ell_L$ with the property 
$\ell_L(C)\subset P\ot \pi^{-1}(Q^\prime )$. 
Now we need to apply the splitting of the left lifted canonical map 
given by $\ell$ (see \eqref{C1}) to derive the formula 
\[     \label{ell_L}
\ell_L:=(\ha_L\ot\ha_R)\circ \ell^\prime + \ell - \ell*(m_P\circ (\ha_L\ot\ha_R)\circ \ell^\prime). 
\]
It is clear that $\ell_L(e)=1\ot 1$ and 
$\ell_L(C)\subseteq P\ot \pi^{-1}(Q^\prime )$. 
A computation similar to \eqref{coli} shows the right colinearity of $\ell_L$. 
Since furthermore $\psi_P(1\ot c)=c\ot 1$ for any $c\in C$ and 
$\tcan =\psi_P\circ \tcan_L$, we obtain 
\[
\tcan(\ell_L(c))=\psi_P\big(\tcan_L(\ell(c))\big) = \psi_P(c\ot 1)=1\ot c. 
\]
Hence $\ell_L$ is a desired strong connection. Plugging it into \eqref{lR} 
instead of $\ell$, we get a strong connection 
\[
\ell_{LR}= (\ha_L\ot \ha_R)\circ \ell^\prime + \ell_L 
- (m_p\circ (\ha_L\ot \ha_R)\circ \ell^\prime)*\ell_L
\]
with the property $\ell_{LR}\subseteq \pi^{-1}(Q^\prime)\ot \pi^{-1}(Q^\prime)$. 
Applying now  Lemma~\ref{L1} ends the proof of this lemma. 
\end{proof}

\subsection{The one-surjective pullbacks of principal coactions
are principal}\label{pbpe}


Our goal now is to show that the subcategory of principal extensions 
is closed under one-surjective pullbacks. Here the right coaction
is the coaction defining a principal extension and the left coaction
 is the
one defined by the inverse of the canonical entwining 
(see~\eqref{leftco}). 
With this structure, principal extensions form a full subcategory
of ${}_e^C\!\mathbf{Alg}_e^C$.
The following
theorem is the main result of this paper generalizing
 the theorem of
\cite{hkmz} on the pullback of surjections of principal comodule
algebras:
\begin{thm}\label{prop-pe}
Let $C$ be a coalgebra, $e\in C$ a group-like element, and
$P$  the pullback of
$\pi_1:P_1\rightarrow P_{12}$ and  $\pi_2:P_2\rightarrow P_{12}$ in
the category ${}_e^C\!\mathbf{Alg}_e^C$ of  unital algebras with 
$e$-coaugmented left and right $C$-coactions.  
If $\pi_1$ or $\pi_2$ is surjective
and both $P_1$ and $P_2$ are principal $e$-coaugmented
$C$-extensions, then also
$P$ is
a  principal $e$-coaugmented $C$-extension.
\end{thm}
\begin{proof}

Without loss of generality, we assume that $\pi_1$ is surjective. 
We first show that $P$ inherits an entwined structure from $P_1$ and $P_2$. 
\begin{lem} \label{2.3}
 Let $\psi_1$ and $\psi_2$ denote the entwining structures of $P_1$ and $P_2$, 
respectively. Then $P$ is an entwined module with an invertible entwining structure 
\[
\psi =\psi_1 \circ (\id\ot\pr_1) + \psi_2 \circ (\id\ot\pr_2). 
 \]
 Here $\pr_1$ and $\pr_2$ are  morphisms of the pullback diagram 
 as in \eqref{A_is_fibre_product}. 
\end{lem}
\begin{proof} 
Our strategy is to construct a bijective map 
$\tilde\psi: C\ot (P_1\times P_2)\ra (P_1\times P_2)\ot C$,  
and to show that it restricts to a bijective entwining on $C\ot P$. 
We put 
\[
\tilde\psi :=\psi_1 \circ (\id\ot\tilde\pr_1) + \psi_2 \circ (\id\ot\tilde\pr_2). 
 \]
The symbols $\tilde\pr_1$ and $\tilde\pr_2$ stand for respective 
componentwise projections. Their restrictions to $P$ yield 
$\pr_1$ and $\pr_2$. It is easy to check that the inverse of 
$\tilde\psi$ is given by 
\[
\tilde\psi^{-1}  = \psi_1^{-1} \circ (\tilde\pr_1\ot \id)+ \psi_2^{-1} \circ (\tilde\pr_2\ot\id)
\]

To show that $\tilde\psi(C\ot P)\subseteq P\ot C$ and 
$\tilde\psi^{-1}(P\ot C)\subseteq C\ot P$, 
we note first that
$P_{12}$ and $\pi_2(P_2)$ are principal by Lemma~\ref{fer}(iv). 
Consequently, their canonical entwinings $\psi_{12}$ 
and $\psi_{\pi_2(P_2)}$ are bijective. 
Furthermore, arguing as in the proof of Lemma~\ref{fer}, we see that 
 $ \psi_{\pi_2(P_2)} = \psi_{12}\!\hsp\upharpoonright_{ C\ot\pi_2(P_2)}$
 and $ \psi_{\pi_2(P_2)}^{-1} = \psi_{12}^{-1}\!\hsp\upharpoonright_{ \pi_2(P_2)\ot C}$. 
 An advantage of having both summands in terms of $\psi_{12}$ is that 
we can apply \eqref{pipsi} to compute 
\begin{align}\nonumber
\big((\pi_1\hsp \circ\hsp  \tilde\pr_1 - \pi_2\hsp \circ\hsp  \tilde\pr_2)\ot \id   \big)\hsp \circ\hsp  \tilde\psi 
&= (\pi_1\hsp \circ\hsp  \tilde\pr_1\hsp  \ot\hsp  \id)\hsp \circ\hsp  
\psi_1\hsp  \circ\hsp  (\id\hsp \ot\hsp\tilde\pr_1)
-(\pi_2\hsp \circ \hsp \tilde\pr_2\hsp \ot \hsp\id)\hsp \circ\hsp  \psi_1 \circ (\id\hsp \ot\hsp\tilde\pr_1)\\
&\hs+\hs(\pi_1\hsp \circ\hsp  \tilde\pr_1\hsp  \ot\hsp  \id)\hsp \circ\hsp         \nonumber
\psi_2\hsp  \circ\hsp  (\id\hsp \ot\hsp\tilde\pr_2)
-(\pi_2\hsp \circ \hsp \tilde\pr_2\hsp \ot \hsp\id)\hsp \circ\hsp  \psi_2 \circ (\id\hsp \ot\hsp\tilde\pr_2)\\
&=(\pi_1\hsp  \ot\hsp  \id)\hsp \circ\hsp  \psi_1\hsp  \circ\hsp  (\id \ot\tilde\pr_1)
-(\pi_2 \ot \id)\hsp \circ\hsp  \psi_2 \circ (\id \ot\tilde\pr_2)  \nonumber\\
&= \psi_{12}\circ (\id\ot \pi_1)\circ (\id \ot \tilde\pr_1) 
-\psi_{\pi_2(P_2)} \circ (\id\ot \pi_2)\circ (\id \ot \tilde\pr_2) \nonumber\\
&= \psi_{12}\circ \big(\id\ot (\pi_1\circ \tilde\pr_1 - \pi_2\circ \tilde\pr_2)  \big ).
\end{align}
Hence $\tilde\psi (C\ot P)\subseteq P\ot C$. Much the same way,
using \eqref{psiinv} instead of \eqref{pipsi}, we show that the bijection 
$\tilde\psi^{-1}(P\ot C)\subseteq C\ot P$. 

It remains to verify that the bijection 
$\psi= \tilde\psi\!\upharpoonright_{ C\ot P}$ is an entwining that makes $P$ an entwined module. 
The former is proven by a direct checking of \eqref{entwining} and \eqref{e2}. 
The latter follows from the fact that $P_1$ and $P_2$ are, respectively, 
$\psi_1$ and $\psi_2$ entwined modules:
\begin{align} \nonumber
\Delta_P(pq) &=\Delta_{P_1}(\pr_1(p)\pr_1(q)) + \Delta_{P_2}(\pr_2(p)\pr_2(q)) 
\nonumber\\
&=\pr_1(p_{(0)})\hs \psi_1(p_{(1)}\ot \pr_1(q)) + \pr_2(p_{(0)})\hs \psi_2(p_{(1)}\ot \pr_2(q))
\nonumber\\
&=\big(\pr_1(p_{(0)}) +\pr_2(p_{(0)}) \big)\big( \psi_1(p_{(1)}\ot \pr_1(q)) +  \psi_2(p_{(1)}\ot \pr_2(q)) \big)
\nonumber\\
&= p_{(0)} \psi(p_{(1)}\ot q ).
\end{align}
This proves the lemma.
\end{proof}

Let $\alpha^1_L$ and $\alpha^1_R$ be a unital left colinear splitting and a unital right 
colinear splitting of $\pi_1$, respectively. Also, let $\alpha^2_R$ be
a right colinear splitting of $\pi_2$ viewed as a map onto $\pi_2(P_2)$. 
Such maps exist by Lemma~\ref{fer}. 
On the other hand, by \cite[Lemma~2.2]{bh04}, since $P_1$ and $P_2$ are principal, 
they admit strong connections $\ell_1$ and $\ell_2$, respectively. 
For brevity, let us introduce the notation
\[     \label{alpha}
\alpha^{12}_L:=\alpha^1_L\circ\pi_2,
\quad
\alpha^{12}_R:=\alpha^1_R\circ\pi_2,
\quad
\alpha^{21}_R:=\alpha^2_R\circ\pi_1\!\upharpoonright_{\hsp\pi_1^{-1}(\pi_2(P_2))}\,,
\quad
L:=m_{P_1}\circ(\alpha^{12}_L\otimes\alpha^{12}_R)\circ\ell_2\,,
\]
where $m_{P_1}$ is the multiplication of $P_1$. 
The situation is illustrated in the following diagram: 
\begin{align}
\label{dia} \\[-18pt]
\nonumber
\xymatrix{ 
 \ \mbox{\mbox{ }\ \ \quad ${}^{\mbox{$C$}}\ \mbox{ }$}\ 
 \ar@<1.5ex>[dd]_L&&&&\ar@<-0.7ex>[lllldd]_{\mathrm{pr}_1} {}^{\mbox{$P$}} \ar@<0.7ex>[rrrrdd]^{\mathrm{pr}_2} &&&&\\
&&&&\ar@{_{(}->}[lllld]<-0.8ex> \ \pi_1^{-1}(\pi_2(P_2)) \ \ar@<1.0ex>[rrrrd]_{\alpha_R^{21}}&&& & \\
\ \mbox{\phantom{$\tilde P$}\ \ \quad $P_1\ \mbox{ }$}\ \ar@{-}[rrd]  && &&   
& &  &&  \ar@<-1.2ex>[llllllll]_{\alpha_L^{12}}\ar@<-0.2ex>[llllllll]^{\alpha_R^{12}}
	\ \ P_2\ .\qquad\mbox{ }\ar@{-}[lld]  \ar@{->>}[dllll]<0.5ex>^{\pi_2} \\
&& {\mbox{\footnotesize$\pi_1$}} \ar@{->>}[rrd] & &  \ar@<0.5ex>[urrrr]^{\alpha^2_R} 
 \makebox[0pt]{\phantom{$\Big($}}\pi_2(P_2) \ar@{_{(}->}[d]<-.2ex>
& & \makebox[0pt]{\phantom{$\big($}}{\mbox{\footnotesize$\pi_2$}} \ar[lld]& &\\
   &&&&\ar@<1.5ex>[uullll]^{\alpha^1_L} \ar@<-1.2ex>[uullll]_{\alpha^1_R}\  \makebox[0pt]{\phantom{$\big($}}P_{12}\   
&& & &
}
\end{align}

Our proves hinges on constructing a strong connection on $P$ out of 
strong connections on $P_1$ and $P_2$. 
Roughly speaking, the basic idea is to take a strong 
connection on $P_2$, induce a strong connection on the 
the common part $P_{12}$, and prolongate it to $P_1$.  
To this end, we check that 
$(\ha_L^{12}  + \id) \ot (\ha_R^{12}  + \id)$ 
is a unital bicolinear map from $P_2\ot P_2$ to $P\ot P$. 
Therefore, as a first approximation for constructing a strong connection on $P$, 
we choose the formula 
\[
\ell_I := \big( (\ha_L^{12}  + \id) \ot (\ha_R^{12}  + \id)\big) \circ \ell_2 .
\]
It is a bicolinear map from $C$ to $P \ot P$ satisfying 
$\ell_I(e)=1\ot 1$ as needed. However, it does not split the lifted canonical map:
\begin{align}
&(\tcan \circ  \ell_I)(c) - 1\ot c   \nonumber\\
&\qquad = \ha_L^{12}(\ell_2(c)^{\ip{1}} )\hs \ha_R^{12}(\ell_2(c)^{\ip{2}})_{(0)} 
   \ot \ha_R^{12}(\ell_2(c)^{\ip{2}})_{(1)} 
   + \ell_2(c)^{\ip{1}}{\ell_2(c)^{\ip{2}}}_{\!(0)} \ot {\ell_2(c)^{\ip{2}}}_{\!(1)} 
   - 1\ot c \nonumber\\
&\qquad =  \ha_L^{12}(\ell_2(c_{(1)})^{\ip{1}})\hs \ha_R^{12}(\ell_2(c_{(1)})^{\ip{2}})
    \ot c_{(2)}   +  (0,1)\ot c - 1\ot c  \nonumber\\
&\qquad = L(c_{(1)}) \ot c_{(2)}  -  (1,0) \ot c \;\in\;  P_1 \ot C. 
\end{align}
We correct it by adding to $\ell_I(c)$ the splitting of 
the lifted canonical map on $P_1\ot P_1$
afforded by $\ell_1$ and applied to 
$(1,0)\ot c - L(c_{(1)})\ot c_{(2)}$:
\begin{align}
\ell_{II}(c)&= \ell_I(c) + \ell_1(c)^{\ip{1}} \ot \ell_1(c)^{\ip{2}} 
   - L(c_{(1)})\hs \ell_1(c_{(2)})^{\ip{1}} \ot \ell_1(c_{(2)})^{\ip{2}}
   \nonumber\\
   &=(\ell_I + \ell_1 - L *\ell_1)(c) . 
\end{align}
The above approximation to a strong connection on $P$ is clearly right colinear. 
Using the fact that $P_1$ is a $\psi_1$-entwined and $e$-coaugmented module, 
we follow the lines of \eqref{coli0}--\eqref{coli} 
to show that $L*\ell_1$ 
is left colinear. 
Hence $\ell_{II}$ is bicolinear. It also satisfies $\ell_{II}(e)=1\ot 1$. However, 
the price we pay for having 
$\ell_{II}(c)^{\ip{1}}\hs {\ell_{II}(c)^{\ip{2}}}_{\!(0)} \ot {\ell_{II}(c)^{\ip{2}}}_{\!(1)} = 1 \ot c$ 
is that the image of $\ell_{II}$ is no longer in $P\ot P$. 

The troublesome term $\ell_1-L*\ell_1$ takes values in $P\ot P_1$. 
Now one would like to compose $\id \ot (\id + \ha_R^{21})$ with 
$\ell_1-L*\ell_1$ to force it taking values in $P\ot P$. 
However, since $\ha_R^{21}$ is defined only on 
$\pi_1^{-1}(\pi_2(P_2))$, we need to replace an arbitrary strong connection 
$\ell_1$ by a strong connection taking values in $P_1\ot \pi_1^{-1}(\pi_2(P_2))$. 
Such a strong connection is provided for us by \eqref{ell_L}:  
\[     \label{tell_L}
\tilde\ell_1:=(\ha_L^{12}\ot\ha_R^{12})\circ \ell_2 + \ell_1 - \ell_1*L. 
\]
Inserting 
$\tilde\ell_1$ in place of $\ell_1$ allows us to apply the correction map 
$\id \ot (\id + \ha_R^{21})$ to obtain 
\[ \label{elll}
\ell_{III} = \ell_I + \big(\id \ot (\id + \ha_R^{21})\big)\circ(\tilde\ell_1-L*\tilde\ell_1). 
\]
To end the proof, let us check that $\ell_{III}$ is indeed a strong connection on $P$. 
First, since $\ell_I(C)\subseteq P\ot P$ and 
$(\id + \ha_R^{21})(\pi_1^{-1}(\pi_2(P_2)))\subseteq P$, 
we conclude that $\ell_{III}$ takes values in $P\ot P$. 
Next, it is bicolinear because $\ha_R^{21}$ is right colinear. 
Also, it is clearly unital. To verify that $\ell_{III}$ splits the canonical map, first 
we note that 
$\tcan \circ (\id \ot  \ha_R^{21})\circ  (\tilde\ell_1-L*\tilde\ell_1)=0$ because 
$m_{P_1\times P_2}(p_1\ot p_2)=0$ for all $p_1\in P_1$ and $p_2\in P_2$. 
Combining this with the fact that 
$\tcan\circ (\ell_1^\prime - L*\ell_1^\prime)$  does not depend 
on the choice of  a strong connection $\ell_1^\prime$, we infer that 
$\tcan\circ \ell_{III}=\tcan\circ \ell_{II}=1\ot \id$. 
Thus $\ell_{III}$ is a strong connection on $P$,  as desired. 
Combining this fact with Lemma~\ref{2.3}  and Lemma~\ref{L1} proves the theorem. 
\end{proof}

Putting the formulas in the proof of Theorem~\ref{prop-pe}  together, we obtain 
the following strong connection on $P$: 
\begin{align}\label{key}
\ell\;=\;&\llp(\alpha^{12}_L+\id)\otimes(\alpha^{12}_R+\id)\lrp\circ\ell_2\\
&+\;(\eta_1\circ\he-L)*\Llp\llp\id\otimes(\id+\alpha^{21}_R)\lrp\circ
\llp\ell_1-\ell_1*L+(\alpha^{12}_L\otimes\alpha^{12}_R)\circ\ell_2\lrp
\Lrp.\nonumber
\end{align}

\section{The pullback picture of the standard quantum Hopf fibration}

Recall that the classical Hopf fibration is a $\rmU$-principal bundle 
given by the maps
\begin{align}
\pi : \S^3=\{(z_1,z_2)\in\C^2\,|\, |z_1|^2+|z_2|^2=1\}
&\lra \S^2\cong \C\mathrm{P}^1, \quad \pi((z_1,z_2))=[(z_1:z_2)], 
\nonumber \\
\S^3\times \rmU&\longrightarrow \S^3,\quad
(z_1,z_2)\triangleleft u= (z_1u,z_2u).
\end{align} 
To unravel the structure of this non-trivial fibration, we
split $\S^3$ into two disjoint parts: 
\[
\S^3=\{(z_1,z_2)\in\C^2\,|\, |z_1|^2<1, |z_2|^2=1- |z_1|^2\}\ 
\cup\  \{(z_1,0)\,|\, |z_1|=1\}.
\]
Note that both sets are invariant under the  $\rmU$-action.
The second set is $\rmU$, and first set is $\rmU$-equivariantly homeomorphic to the interior of the solid torus
 $\rmD\times \rmU$ equipped with the diagonal action. (Here 
$\rmD=\{z\in\C\,|\,|z|\leq 1\}$.) 
By an appropriate $\rmU$-equivariant gluing of
the boundary torus of $\rmD\times\rmU$ with $\rmU$, we recover
$\S^3$ with its   $\rmU$-action: 
\begin{align}\mbox{ } \label{diagram} \\[-18pt]
\nonumber
\mbox{
\xymatrix@=-1mm@R=2mm{ 
&&\S^3&&\\
 \phi_1(z,v)=(z, v\sqrt{1\hsp -\hsp |z|^2})  \quad&&&&\quad 
\phi_2(u)=(u,0)\\
&\rmD\times \rmU\ar[ruu]^{\phi_1}
&&\rmU\ar[luu]_{\phi_2}& \\
(\iota,\id)(u,v)=(u,v)\quad&&&&\quad\pr_1(u,v)=u\,.\\
&&\rmU\times \rmU\ar[ruu]_{\pr_1}\ar[luu]^{(\iota,\id)}&&
}
}
\end{align}

However, to view 
$\rmD\times \rmU$ as
a trivial $\rmU$-principal bundle, we need to gauge the diagonal 
action to the action on the right slot. This is achieved with
the help of the following homeomorphism intertwining these two actions:
\begin{align}
&\Psi : \rmD\times \rmU \lra \rmD\times \rmU,\quad \Psi(x,v):=(xv,v),
\quad \Psi(x,vu)=\Psi(x,v)\triangleleft u.
\end{align}
Let us denote the restriction of $\Psi$
to $\rmU\times\rmU$ by the same symbol. Now we can extend 
the above pushout diagram to the commutative diagram
\begin{align}\mbox{ } \label{dg} \\[-18pt]
\nonumber
\mbox{
\xymatrix@=-1mm@R=10mm{ 
&&&&\S^3&&&&\\
\rmD\times \rmU\ar[rrr]^{\Psi}&&&\rmD\times \rmU\ar[ru]^{\phi_1}
&&\rmU\ar[lu]_{\phi_2} &&&\ar[lll]_{\text{id}} \rmU\\
&\mbox{\ \quad \ }&\mbox{\ \quad \ }
&&\rmU\times \rmU\ar[ru]_{\pr_1}\ar[lu]^{(\iota,\id)}
&&\mbox{\ \quad \ }&\mbox{\ \quad \ }&\\
&&&&\ar[lllluu]^{(\iota,\id)}\rmU\times \rmU\ar[u]_{\Psi}\ar[rrrruu]_m\;,&&&&
}
}
\end{align}
where $m$ is the multiplication map. The outer diagram is again a
pushout diagram of $\rmU$-spaces, but now its defining $\rmU$-spaces
are trivial $\rmU$-principal bundles. It is the outer pushout diagram
that we shall use to analyse a noncommutative deformation of the
Hopf fibration.

\subsection{Pullback comodule algebra}\label{fpa}

We consider the  tensor products
$\lP_1:=\T\otimes \OU$, \mbox{$\lP_2=\C\otimes \OU\cong \OU$} and
$\lP_{12}:=\CS\otimes\OU$. These algebras are right $\OU{}$-comodule algebras 
for the coaction $x\otimes {\V}^{N}\mapsto x\otimes {\V}^{N}\otimes {\V}^{N}$,  $N\in\Z$.
Moreover, $\lP_1$ and $\lP_2$ are trivially principal with strong connections 
$\ell_i:\OU{}\ra \lP_i\otimes \lP_i$ given by
$\ell_i({\V}^{N})=(1\otimes {\V}^{N*})\otimes (1\otimes {\V}^{N})$, \,$i\hsp =\hsp 1,2$.
To construct a pullback of $\lP_1$ and $\lP_2$, we define 
the following morphisms of right $\OU{}$-comodule algebras:
\begin{align}
 &\pi_1:\T \otimes\OU \lra \CS\otimes\OU,\quad \pi_1(t\otimes w)=\sigma(t)\otimes w,\\
 &\pi_2:\OU\lra \CS\otimes\OU,\quad \pi_2(w)=\Delta(w).       
\label{WPhi}
\end{align}
Then the fibre product
$\PW:= \T\otimes \OU \times_{(\pi_1,\pi_2)} \OU$ defined by the pullback diagram
\begin{equation}    \label{fpSU2}
\xymatrix{
& \makebox[48pt][c]{$\T\otimes\OU \,{\underset{(\pi_1,\pi_2)}{\times}}\, \OU$}
\ar[dl]_{\mathrm{pr}_1} \ar[dr]^{\mathrm{pr}_2}& \\
\T \otimes\OU\ar[rd]_{\pi_1} &    &
\OU \ar[ld]^{\pi_2}\\
 &  \CS\otimes\OU&   }
\end{equation}
is a right $\OU$-comodule algebra. By Proposition~\ref{prop-pe}, it is principal. 

Furthermore, define unital respectively 
left colinear and right colinear splittings of  $\pi_1$  by 
\[
  \alpha^{1}_L(f\ot {\V}^{N}) = \alpha^{1}_R(f\ot {\V}^{N}) =T_f\ot {\V}^{N}, \quad 
  N\in\Z.
\]
Here $f\in \CS$ 
and $T_f$ denotes the Toeplitz operator with symbol $f$. 
In particular, $T_{u^N}=\S^N$ and $T_{u^{*N}}=\S^{*N}$. 
A right colinear splitting of the map  
$\pi_2:\OU\rightarrow \pi_2(\OU)$ is given by 
\[
\alpha^{2}_R ({\V}^{N}\ot {\V}^{N})= {\V}^{N},\quad N\in\Z. 
\] 
Inserting the definitions of  $\alpha^{1}_L$, $\alpha^{i}_R$ and $\ell_i$,  \,$i\hsp =\hsp 1,2$, 
into \eqref{alpha} and \eqref{key}, we obtain the following strong connection on $\PW$: 
\begin{align}
\ell({\U}^{N})&=({\S}^{*N}\otimes {\V}^{*N},{\V}^{*N} ) \otimes({\S}^{N}\otimes {\V}^{N},{\V}^{N}),
\label{luN} \\
 \ell({\U}^{*N})&=({\S}^{N}\otimes {\V}^{N},{\V}^{N} ) \otimes({\S}^{*N}\otimes {\V}^{*N},{\V}^{*N}) \notag \\
& \quad +((1-{\S}^{N}{\S}^{*N})\otimes {\V}^{N},0)\otimes
((1-{\S}^{N}{\S}^{*N})\otimes {\V}^{*N},0), \hspace{24pt} N\in\N.
\label{lu*N}
\end{align}

Note next that,
by construction, we have
\[
\PW= \Big\{ \mbox{$\sum_k$} ( t_k\otimes {\V}^{k}, \alpha_k {\V}^{k})\in \big(\T\otimes\OU\big) \times \OU \;\Big|\; 
\sigma(t_k) =  \alpha_k{\U}^{k} \Big\},
\]
where $\alpha_k\hsp\in\hsp\C$.
For ${}_{\C}\Delta(1)= \U^N\ot 1$, let 
\begin{equation}                                                             \label{LN}
L_N\,:= \, \PW\underset{\cO(\mathrm{U}(1))}{\square}\C \, \cong \,  \{ p\in \PW \;|\; \Delta_\PW(p)=p\otimes {\V}^{N}\}.
\end{equation}
Then $L_0 = \PW^{\co\cO(\mathrm{U}(1))}$,
each $L_N$ is a left $\PW^{\co\cO(\mathrm{U}(1))}$-module and
$\PW= \bigoplus{}_{N\in\Z}\, L_N$.
From
\[
\Delta_\PW\big(\sum_k ( t_k\hsp\otimes\hsp {\V}^{k}, \alpha_k {\V}^{k})\big)\hsp=\hsp
\sum_k ( t_k\hsp\otimes\hsp {\V}^{k}, \alpha_k {\V}^{k})\otimes {\V}^{k},
\]
it follows that
\[
L_N= \Big\{ ( t\otimes {\V}^{N}, \alpha {\V}^{N})\in \big(\T\otimes\OU\big) \times \OU \;\Big|\; 
\sigma(t) = \alpha\hs {\U}^{N}  \Big\}.
\]

The next proposition shows that $L_0\cong \T \times_{(\sigma,1)}\C$ is isomorphic
to the $C^*$-algebra of the standard Podle\'s sphere and that
%
\[   \label{isom}
L_N\cong \T \times_{({\U}^{-N}\sigma,1)}\C,
\]
where $\T \times_{({\U}^{N}\sigma,1)}\C$ is given by the pullback diagram
\begin{equation}\label{fpPodles}
\xymatrix{
& \T \,{\underset{({\U}^{-N}\sigma ,1)}{\times}}\, \C\ar[dl]_{\mathrm{pr}_1} \ar[dr]^{\mathrm{pr}_2}& \\
\T \ar[d]_{\sigma} &    &  \C \ar[d]^{\alpha \mapsto \alpha 1}\\
\CS \ar[rr]_{f\mapsto {\U}^{-N}f} & & \CS.   }
\end{equation}

\begin{prop}\label{CPodles}
The fibre product
$\T {\times}_{(\sigma,1)} \C$ is isomorphic to
the $C^*$-algebra $\CSq$, and
$L_N$  is isomorphic to  $\T \times_{({\U}^{-N}\sigma,1)}\C$
as a left $\CSq$-module with respect to
the left $\CSq$-action on $\T \times_{({\U}^{-N}\sigma,1)}\C$
given by $(t,\alpha)\cdot(h,\beta)=(th,\alpha\beta)$.
\end{prop}
\begin{proof}
For $N=0$, the mappings
$\T\ni t \mapsto \sigma(t)\in\CS$
and $\C\ni \alpha \mapsto \alpha 1\in\CS$  are  morphisms
of $C^*$-algebras, so that
$\T{\times}_{(\sigma,1)} \C$ is a $C^*$-algebra.
Next, recall that $\CSq\cong \cK(\lN) \oplus\C$ (see~\eqref{g}), and define
\begin{align}                                                             \label{Psi}
&\phi :\T\,{\underset{(\sigma,1)}{\times}}\hs\C\, \longrightarrow \,\cK(\lN) \oplus\C, \quad
\phi(t,\alpha)=t,\\                                                     \label{Phi}
&\phi^{-1} : \cK(\lN) \oplus\C\,\longrightarrow \,\T\,{\underset{(\sigma,1)}{\times}}\hs\C, \quad
\phi^{-1}(k+\alpha)=(k+\alpha\hs,\hs\alpha).
\end{align}
Clearly, $\phi:\T{\times}_{(\sigma,1)}\C\rightarrow \B(\lN) $ is a morphism of $C^*$-algebras. Since
$\phi(t,\alpha)=(t-\alpha)+\alpha$, and $\s(t-\alpha)=0$ by the pullback
diagram \eqref{fpPodles}, it follows from the short exact 
sequence~\eqref{sesToeplitz}
that $t-\alpha\in\cK(\lN)$, so that
$\phi(t,\alpha)\in \cK(\lN) \oplus\C$.
One easily sees that $\phi^{-1}$ is its inverse 
so that $\T{\times}_{(\sigma,1)}\C\cong \cK(\lN) \oplus\C$.

The fact that $\T {\times}_{({\U}^{-N}\sigma,1)} \C$ with the given $\CSq$-action
is a left $\CSq$-module
follows from the discussion preceding the pullback diagram \eqref{fpMod}
with the free rank 1 modules $E_1=\T$, $E_2=\C$ and
$\pi_{1\ast}E_1=\pi_{2\ast}E_2=\CS$.
Obviously,
$L_N \ni ( t\otimes {\V}^{N}, \alpha {\V}^{N})\mapsto (t,\alpha)\in \T {\times}_{({\U}^{-N}\sigma,1)}\C$
defines an isomorphism of left $\CSq$-modules.
\end{proof}

\subsection{Equivalence of the pullback and standard constructions}

Let us view $\mathrm{U}(1)$ as a compact quantum group. 
We consider its $C^*$-algebra $\CS$ of all continuous 
function together with the obvious coproduct, counit and 
antipode given by 
$\Delta(f)(x,y)=f(xy)$, $\vare(f)=f(1)$ and
$S(f)(x)=f(x^{-1})$, respectively. 
Furthermore, let $\barot$ stand for the 
completed tensor product of 
$C^*$-algebras.  In our case it is unique because of 
the nuclearity of the  involved $C^*$-algebras.  

Now let $\pi_2: \CS\ra \CS\barot\CS$ be given by the coproduct, i.e., 
$\pi_2(f)(x,y)=(\Delta f)(x,y)= f(xy)$, and let
$\sigma\otimes \id$ denote the tensor product of the
symbol map $\sigma:\T\ra\CS$ and $\id:\CS\ra\CS$.
Then $\barP := \T\barot \CS \times_{(\pi_1 ,\pi_2)}\CS$
is defined by the pullback diagram
\begin{equation}                                             \label{barP}
\xymatrix{
& \makebox[48pt][c]{$\T\barot\CS \,{\underset{(\pi_1 ,\pi_2)}{\times}}\, \CS$}
\ar[dl]_{\mathrm{pr}_1} \ar[dr]^{\mathrm{pr}_2}& \\
\T \barot\CS\ar[rd]_{\pi_1=\sigma\otimes\id\ \ } &    &  \CS \ar[ld]^{\ \pi_2=\Delta}\\
 &\CS\barot\CS\ & .  }
\end{equation}

With the $\CS$-coaction given by the coproduct $\Delta$
on the right tensor factor $\CS$, $\pi_1$ and $\pi_2$ are morphisms in
the category of right $\CS$-comodule $C^*$-algebras. 
Equivalently, we can view this diagram as a diagram in the category of 
$\mathrm{U}(1)$-$C^*$-algebras (see Section \ref{Pe}).
Therefore, 
$\barP $ inherits the structure of a right $\mathrm{U}(1)$-$C^*$-algebra.
To determine the Peter-Weyl comodule algebra  $\PWU(\barP)$,  
we first note that $\OU$ is 
the dense Hopf *-subalgebra of $\CS$ 
spanned by the matrix coefficients of the irreducible unitary corepresentations. 
Using the counit $\vare:\CS \ra \C$ and the fact that the 
Peter-Weyl functor commutes with taking pullbacks, 
we easily conclude that 
$\PWU(\barP)= \T\otimes \OU \times_{(\pi_1,\pi_2)} \OU$, 
so $\PWU(\barP)$ is the comodule algebra $\PW$ of 
Section~\ref{fpa}. 
 
Consider next the
\srep\ of $\SUq$ on $\lN$ given by \cite{vs88}
\begin{align}\label{repSUq}
\begin{split}
&\rho(\ha)e_n = (1-q^{2n})^{1/2} e_{n-1}, 
\quad \rho(\hb)e_n= -q^{n+1}e_n,\\
&\rho(\hc)e_n= q^{n}e_n, \quad \rho(\hd)e_n 
= (1-q^{2(n+1)})^{1/2} e_{n+1}.
\end{split}
\end{align}
Note that $\rho(\hb), \, \rho(\hc)\in\cK(\lN)$.
Comparing $\rho$ with the representation $\mu$ of $\Dq$ from \eqref{mu}, one readily
sees that $\rho(\SUq)\subseteq \T$. Furthermore, 
the symbol map $\s$ yields $\s(\rho(\hb))=\s(\rho(\hc))=0$.  
Using an appropriate diagonal compact operator $k$, we also obtain 
\[\s(\rho(\ha))=\s(\rho(\ha)-\S^* )+ \s(\S^*)= 
\s(k\S^* )+ \s(\S^*) ={\U}^{-1} ,\quad \s(\rho(\hd))=\s(\rho(\ha))^*={\U}. 
\]
Thus 
we obtain a $\rmU$-equivariant *-algebra morphism 
$\iota:\SUq\ra \PW$ by setting
\begin{equation}                                                        \label{i}
 \iota(\ha) = (\rho(\ha) \otimes {\V},{\V}),\quad \iota(\hc) = (\rho(\hc) \otimes {\V},0).
\end{equation}
One easily checks that the image of a Poincar\'e-Birkhoff-Witt basis of $\SUq$
remains linearly independent, so that $\iota$ is injective and we can consider $\SUq$
as a subalgebra of $\PW$. In particular, we have $\iota(M_{N})\subseteq L_N$ as
left $\Sq$-modules. (See Section~\ref{sPs} and Section~\ref{fpa} 
for the definitions of $M_N$
and $L_N$, respectively.)

The main objective of this section is to establish an $\rmU$-$C^*$-algebra
isomorphism between $\CSUq$ and $\barP $.
The universal $C^*$-algebra $\CSUq$ of $\SUq$ has been studied in
\cite{MNW} and \cite{W}.
Here we shall use the fact from \cite[Corollary 2.3]{MNW} that
a faithful *-representation $\hat\rho$ of $\CSUq$ on
the Hilbert space $\lN\barot\ell_2(\Z)$ is given by 
\begin{equation}                                         \label{concrete}
\hat\rho(\ha)(e_n\otimes b_k)= (1-q^{2n})^{1/2} e_{n-1}\otimes b_{k-1},\quad
\hat\rho(\hc)(e_n\otimes b_k)= q^{n}e_{n}\otimes b_{k-1},
\end{equation}
where $\{e_n\}_{n\in{\N}}$ and $\{b_k\}_{k\in\Z}$ denote the standard bases of
$\lN$ and $\ell_2(\Z)$, respectively. 
To compare \eqref{concrete} with \cite[Corollary 2.3]{MNW}, one has to apply 
the unitary transformation 
\[T: 
\lN\barot\ell_2(\Z)\ra \lN\barot\ell_2(\Z),
\quad T(e_n\otimes b_k)= e_n\otimes b_{k-n}.
\]
A right $\CS$-coaction on $\CSUq$ is
given by $(\id\ot\bar\pi)\circ \Delta$, where $\Delta$ denotes the coproduct of
the compact quantum group $\CSUq$ and $\bar\pi$ is the extension of the
Hopf *-algebra surjection $\pi:\SUq\ra \OU$ defined in \eqref{Hsur} 
to $\CSUq$. Using the faithfulness of $\hat\rho$, we can transfer $\bar\pi$ 
to $\hat\rho(\CSUq)$. 
In \cite{MNW}, it is shown that $\bar\pi$ gives rise to the 
$C^*$-algebra extension 
\begin{equation}                                             \label{C-ext}
\xymatrix{
0 \ar[r] & 
\ \cK(\lN)  \barot\CS\ \ar@{^{(}->}[r] &
\ \hat\rho(\CSUq)\ \ar[r]^{\quad \bar\pi}  &  \ \CS\  \ar[r] & 0 \,.  }
\end{equation}

\begin{thm}\label{thm}
The $\rmU$-$C^*$-algebras $\CSUq$ and $\barP $ are isomorphic.
\end{thm}
\begin{proof}
First note that $\ker (\mathrm{pr}_1)  = 
\{ (0,y)\in \barP \,|\, \pi_2(y)= \Delta(y) =0\}=\{0\}.$
Hence we can identify $\barP $ with the image of $\pr_1$ in $\T \barot\CS$. 
We will prove the theorem by applying the five lemma to the following 
commutative diagram of $\rmU$-$C^*$-algebras: 
\begin{equation}                                             \label{5lem}
\xymatrix{
0 \ar[r] & 
\ \cK(\lN)  \barot\CS\ \ar[d]^{\id} \ar@{^{(}->}[r] &
\ \hat\rho(\CSUq)\ \ar[r]^{\bar\pi} \ar[d]^\tau &  \ \CS\ \ar[d]^{\id}  \ar[r] & 0 \mbox{\hs\phantom{.}} \\
0 \ar[r] & 
\ \cK(\lN) \barot\CS\ \ar@{^{(}->}[r] &\ \pr_1(\barP )\  \ar[r]^{\omega}  & 
\ \CS\ \ar[r] & 0\;.   }
\end{equation}
To define $\tau$, 
we realize  $\CS$ as a concrete $C^*$-algebra of
bounded operators on $\ell_2(\Z)$ by setting ${\V}(b_k)=b_{k-1}$.
Then 
\begin{equation}                                               \label{hrho}
 \hat\rho(\ha) = \rho(\ha)\otimes {\V}
 = \pr_1(\rho(\ha)\otimes {\V},\V)\in \pr_1(\barP ), \quad
  \hat\rho(\hc)= \rho(\hc)\otimes {\V}
 = \pr_1(\rho(\hc)\otimes {\V},0)\in \pr_1(\barP ).
\end{equation}
Since $\CSUq$ is generated by $\ha$ and $\hc$, 
we take $\tau$ to be the inclusion $\hat\rho(\CSUq)\subset \pr_1(\barP )$. 

We define the $\rmU$-$C^*$-algebra morphism $\omega$ by 
\begin{equation}
\omega : \pr_1(\barP ) \lra \CS, \quad \omega := (\vare\circ\sigma)\otimes\id, 
\end{equation}
where $\vare$ denotes the counit of $\CS$.
The surjectivity of $\omega$ follows from 
$u^k = \omega(\rho(\ha^k)\otimes \V^k)$ and 
$u^{-k} = \omega(\rho(\ha^{*k})\otimes \V^{*k})$ for all $k\in\N$ and taking the closure 
of $\OU$ in $\CS$. 

To prove the exactness of the lower row, note that 
$\cK(\lN)  \barot\CS = \ker (\sigma) \barot \CS\subset \ker (\omega)$. 
Now let 
$f\in \pr_1(\barP )\setminus \ker (\sigma) \barot \CS$. Then 
$(\sigma\otimes\id)(f)\neq 0$.  By the commutative diagram \eqref{barP}, 
there exists a non-zero element $g\in \CS$ such that 
$(\sigma\otimes\id)(f)= \Delta(g)$. 
Hence $\omega(f)=  (\vare\ot\id)\circ \Delta (g)=g\neq 0$ which proves 
$\ker(\omega)= \cK(\lN)  \barot\CS$.  

It remains to show that the diagram \eqref{5lem} is commutative. This is clear for the left 
part since $\tau$ is just the inclusion. The commutativity of the right part follows from 
\begin{align}
&
\omega\Big( \tau \big(\hat\rho(\ha)\big)\Big)=
\vare\Big(\sigma\big(\rho(\ha)\big)\Big)\ot u=\vare( u)u= u =\bar\pi(\hat\rho(\ha)), \\
&\omega\Big( \tau \big(\hat\rho(\hc)\big)\Big)=
\vare\Big(\sigma\big(\rho(\hc)\big)\Big)\ot u= 0=\bar\pi(\hat\rho(\hc)) 
\end{align}
by taking limits since $\hat\ha$ and $\hat\hc$ generate $\CSUq$. Therefore, by the five lemma, 
$\tau$ is an isomorphism of $\rmU$-$C^*$-algebra. 
\end{proof}

By the final remark of Section~\ref{Pe}, 
we conclude from Theorem \ref{thm} that the Peter-Weyl comodule
algebras $\PWU(\CSUq)$ and $\PW$ are isomorphic. We use this
isomorphism to identify associated projective
modules.
For $N\in\Z$ and the left \OU-coaction on $\C$ given by 
${}_{\C}\Delta(1)= \U^N\ot 1$, we
define a ``completed" version of $M_N$ (see \eqref{MsubN}):
\[\label{projN}
\bar M_N:=  \PWU(\CSUq) \underset{\mathcal{O}(\mathrm{U}(1))}{\square}\C     
=\{ p\in\PWU(\CSUq) \;|\; ((\id\ot\bar\pi)\circ \Delta)(p)=p\ot {\U}^{N}\}.
\]
Now it follows from Equation~\eqref{LN} that 
$\bar M_N \cong L_N$. 
Applying the same arguments as at the end of Section~\ref{sPs}, 
we infer that $\bar M_N\, \cong\, \CSq^{N+1} E_N\,$,  
with $E_N$ being the projection matrix of Equation~\eqref{EN}. 
Taking advantage of these isomorphisms of modules, we prove:
\begin{lem}\label{project}
Identifying $\CSq$ with $\cK(\lN) \oplus\C$, 
we obtain the following isomorphisms of left $\CSq$-modules: 
\begin{align}
&\CSq^{N+1} E_N \cong \CSq\hs p_N,  & 
p_{N}&:= {\S}^{N}{\S}^{N*},  & & N\geq 0, \\
&\CSq^{|N|+1} E_N \cong \CSq^2 p_N, &
p_{N}&:= \left(
  \begin{array}{cc}
  1 &  0 \cr
  0 & 1-{\S}^{|N|}{\S}^{|N|*}
  \end{array}
  \right), && N<0.
  \label{pn-}
\end{align}
\end{lem}
\begin{proof} 
We apply \eqref{ie} to construct  projections 
$P_{N}$, $N\in\mathbb{Z}$,
from the strong connection given in \eqref{luN} 
and~\eqref{lu*N}.  
For $N<0$, we obtain
\begin{align}
(P_N)_{11} &= ({\S}^{*|N|}\otimes {\V}^{N},{\V}^{N})({\S}^{|N|}\otimes {\V}^{|N|},{\V}^{|N|} ) = (1\ot 1,1), \\
(P_N)_{12} & =(P_N)_{21}^*= ({\S}^{*|N|}\otimes {\V}^{N},{\V}^{N})
((1-{\S}^{|N|}{\S}^{*|N|})\otimes {\V}^{|N|},0)=0, \\
(P_N)_{22} & = ((1-{\S}^{|N|}{\S}^{*|N|})\otimes {\V}^{N},0)((1-{\S}^{|N|}{\S}^{*|N|})
\otimes {\V}^{|N|},0) = ((1-{\S}^{|N|}{\S}^{*|N|})\otimes 1,0). 
\end{align}
Analogously, for $N\geq0$, we get
\[
(P_{N})_{11} = ({\S}^{N}\otimes {\V}^{N},{\V}^{N})({\S}^{*N}\otimes {\V}^{*N},{\V}^{*N} )= 
({\S}^{N}{\S}^{*N}\otimes 1,1).
\]
Finally, applying the isomorphism \eqref{Psi} componentwise to 
$P_N$, $N\in\mathbb{Z}$,
yields the result.
\end{proof}

The projections $p_N$ of Lemma \ref{project} can also be obtained 
from the odd-to-even construction in Section~\ref{sec-mv}. 
First let $N<0$. Since $L_N\cong T {\times}_{({\U}^{-N}\sigma,1)} \C$
(see~\eqref{isom}),   
we can apply the formula \eqref{p} by taking
$E_1=\T$, $E_2=\C$,
$\pi_{1\ast}E_1=\pi_{2\ast}E_2=\CS$, and choosing 
$h$ in \eqref{fpMod} to be the isomorphism given by the
multiplication with ${\U}^{|N|}$. 
As the symbol map $\s$ applied to ${\S}$ is ${\U}$
(see~\eqref{sesToeplitz}), we can lift
${\U}^{|N|}$ and its inverse ${\U}^{-|N|}$  to
${\S}^{|N|}$ and ${\S}^{|N|*}$ respectively.
Inserting  ${\Cc}= {\S}^{|N|*}$ and ${\Dd}\hsp=\hsp {\S}^{|N|}$ into \eqref{p}
gives $\T {\times}_{({\U}^{-N}\sigma,1)} \C\cong (\T {\times}_{(\sigma,1)} \C)^{2}\hs Q_{N}$, where
\begin{equation}                                                    \label{P-N}
Q_{N}= \left(
  \begin{array}{cc}
  (1,1) &  (0,0) \cr
  (0,0) & (1-{\S}^{|N|}{\S}^{|N|*},0)
  \end{array}
  \right).
\end{equation}
Finally, 
applying the isomorphism \eqref{Psi} yields the projection 
in~\eqref{pn-}.
Similarly, for $N\geq 0$, we insert ${\Cc}={\S}^{N}$ and ${\Dd}={\S}^{N*}$ into \eqref{p}. 
Since ${\S}^{N*}{\S}^{N}=1$, we obtain
$\T {\times}_{({\U}^{-N}\sigma,1)} \C\cong (\T {\times}_{(\sigma,1)} \C)^{2}Q_{N}$ with
\begin{equation}                                                     \label{PN}
 Q_{N}=
  \left(
  \begin{array}{cc}
  ({\S}^{N}{\S}^{N*},1 ) & (0, 0) \cr
  (0, 0) & (0, 0 )
  \end{array}
  \right),
\end{equation}
which is equivalent to
$\T {\times}_{({\U}^{-N}\sigma,1)} \C\cong  \CSq\hs {\S}^{N}{\S}^{N*}$.

\subsection{Index pairing}

Recall that for a $C^*$-algebra $A$, a projection $p\in\Mat_n(A)$,
  and *-representations $\rho_+$ and $\rho_-$ 
of $A$ on a Hilbert space $\H$ such that
$[(\rho_+,\rho_-)]\in {\K}^{0}(A)$ (e.g.\ see \cite[Chapter~4]{c-a94}), 
one has the following: 
If the operator $\tr_{\Mat_n} (\rho_+-\rho_-)(p)$ is trace class, 
then the formula
\begin{equation}\label{CCpair}
\langle [(\rho_+,\rho_-)],[p]\rangle
=\tr_\H (\tr_{\Mat_n} (\rho_+-\rho_-)(p))
\end{equation}
yields a pairing between ${\K}^{0}(A)$ and ${\K}_{0}(A)$.

In this section, we compute the pairing between the $\K_0$-classes of the
projective $\CSq$-modules describing quantum line bundles and the two 
generators of ${\K}^{0}(A)$.
By Lemma~\ref{project},  
we can take the projections $p_{N}$ as representatives of respective
$\K_0$-classes.
Their simple form makes it very easy to compute
the index pairing.

\begin{thm}\label{CCP}
Let $\bar{M}_N$ be the associated modules of~\eqref{projN}, and let 
$[(\id,\vare)]$ and $[(\vare,\vare_0)]$ denote the generators of 
${\K}^{0}(\CSq)$
given in Section~\ref{sPs}.
Then, for all $N\in\Z$,
\begin{equation}
\langle [(\vare,\vare_0)],[\bar{M}_N]\rangle=1,
\qquad  \langle [(\id,\vare)],[\bar{M}_N]\rangle=-N.
\end{equation}
\end{thm}
\begin{proof}
Let $N\geq 0$. Then $p_N={\S}^{N}{\S}^{N*}=({\S}^{N}{\S}^{N*}-1)+1$, so that
$\vare(p_N)=1$ and $\vare_0(p_N)={\S}{\S}^{*}$. Furthermore, since
for any $N\in\N\!\setminus\!\{0\}$, the image of the projection
$1-{\S}^{N}{\S}^{N*}$ is  
$\lin\{e_0,\ldots,e_{N-1}\}\subset\lN$, the projection
$1-{\S}^{N}{\S}^{N*}$ is  trace class. Moreover, with the help
of Lemma~\ref{project} and Formula~\eqref{CCpair}, it implies that
\begin{align}
&\langle [(\vare,\vare_0)],[\bar{M}_N]\rangle=\tr_{{\lN}} (\vare-\vare_0)(p_N)
=\tr_{{\lN}}(1-{\S}{\S}^{*})=1,\\
&\langle [(\id,\vare)],[\bar{M}_N]\rangle=\tr_{{\lN}} (\id-\vare)(p_N)
=\tr_{{\lN}}({\S}^{N}{\S}^{N*}-1)=-N.
\end{align}

For $N<0$, we have $\tr_{\Mat_2}(p_N)=2-{\S}^{|N|}{\S}^{|N|*}=2-p_{|N|}$.
Combining this with  the above index pairing for $p_{|N|}$,
the formulas $(\vare-\vare_0)(2)=2(1-{\S}{\S}^{*})$
and $(\id-\vare)(2)=0$, and  \eqref{CCpair}, we obtain
\begin{align}
\langle [(\vare,\vare_0)],[\bar{M}_N]\rangle & = \tr_{{\lN}} (\vare - \vare_0)(2 - p_{|N|})
 = \tr_{{\lN}}(1 - {\S}{\S}^{*}) = 1,\\
\langle [(\id,\vare)],[\bar{M}_N]\rangle & = \tr_{{\lN}} (\id - \vare)(2 - p_{|N|})
 = \tr_{{\lN}}(1 - {\S}^{|N|}{\S}^{|N|*}) = -N.
\end{align}
This completes the proof.
\end{proof}

The above theorem agrees with the classical situation. Indeed,
the pairing $\langle [(\vare,\vare_0)],[\bar{M}_N]\rangle$ yields 
the rank of the line bundles,
and $\langle [(\id,\vare)],[\bar{M}_N]\rangle$
computes the winding number of
 the map ${\U}^{-N}: \mathrm{S}^1\ra\mathrm{S}^1$.

\section{Examples of piecewise principal coalgebra coactions}

We begin by recalling the piecewise structure \cite{hkmz} of  a 
noncommutative join construction proposed by \cite{dhh}. 
Then we instantiate it to $\mathrm{SU}_q(2)$ to obtain a 
quantum instanton bundle 
$\mathrm{S}^7_q\rightarrow \mathrm{S}^4_q$  \cite{p-mj94} 
as a piecewise trivial principal comodule algebra. 
A key step is then to replace $\mathrm{SU}_q(2)$ by 
quotienting the Hopf algebra $\SUq$ by a coideal 
right ideal $(\cO(\mathrm{S}_{q,s}^2)\cap \mathrm{ker}\hs \vare)\SUq$ 
provided by the generic Podle\'s quantum spheres $\mathrm{S}_{q,s}^2$, 
$s\neq 0$ \cite{p-p87}.
The quotient coalgebra is isomorphic with $\OU$ \cite{MS}. 
Applying our main theorem, we will prove that the induced 
right coaction of $\OU$ is principal.

\subsection{Piecewise principal coactions from a noncommutative
join construction}

Let $\bar H$ be the $C^*$-algebra of a compact quantum group and 
$H$ its Peter-Weyl Hopf algebra \cite{w-sl87,w-sl98}. 
We take the algebra of norm continuous functions 
$C([a,b],\bar H)$ from a closed interval $[a,b]$ 
to the $C^*$-algebra $\bar H$, and 
define 
\begin{eqnarray}
&& P_1:=\{f \in C([0,\mbox{$\frac{1}{2}$}],\bar H) \otimes H\,|\,
	f(0) \in \Delta(H)\},\\ 
&& P_2:=\{f \in C([\mbox{$\frac{1}{2}$},1],\bar H) \otimes H\,|\,
	f(1) \in \mathbb{C} \otimes H\}.
\end{eqnarray} 
Here 
we identify elements of 
$C([a,b],\bar H) \otimes H$ with functions 
$[a,b] \rightarrow \bar H \otimes H$. The $P_i$'s 
are right 
$H$-comodule algebras for the coaction 
$ \Delta_{P_i}=\mathrm{id}_{C([a_i,b_i],\bar H)} \otimes \Delta$, 
where $\Delta$ stands for the coproduct of $H$. 
The subalgebras of coaction invariants can be identified with 
\begin{eqnarray}
&& B_1:=\{f \in C([0,\mbox{$\frac{1}{2}$}],\bar H)\,|\,f(0) \in \mathbb{C}\}, 
		  \nonumber\\ 
&& B_2:=\{f \in C([\mbox{$\frac{1}{2}$},1],\bar H)\,|\,f(1) \in \mathbb{C}\}.
		  \nonumber 
\end{eqnarray}

The comodule algebra $P_2$ is evidently the same as 
$B_2\ot H$. Unlike $P_2$, the comodule algebra $P_1$ 
does not coincide with $B_1\ot H$. However, there is 
a cleaving map 
$j:H\rightarrow P_1$ by $j(h)(t):= \big(t\mapsto h_{(1)}\big)\ot h_{(2)}$, 
that is, $j(h)(t):=\Delta(h)$ for all $t\in[0,\mbox{$\frac{1}{2}$}]$. 
Since $j$ is an algebra homomorphism, 
it identifies the comodule algebra $P_1$ with a 
smash product $B_1\# H$. 

Now one can define the noncommutative join of $\bar H$ 
as the pullback right $H$-comodule algebra 
\[
	P:=\{(p,q) \in P_1 \oplus P_2\,|\,
	\pi_1(p)=\pi_2(q)\}
\]
given by 
the evaluation maps 
\[
\pi_1:=\mathrm{ev}_{\frac{1}{2}}\ot \id  : P_1 \rightarrow P_{12}:=\bar H\ot H,\quad
	\pi_2:=\mathrm{ev}_{\frac{1}{2}}\ot \id : P_2 \rightarrow P_{12}:=\bar H\ot H,
\]
where $\mathrm{ev}_{t}$ is given by the evaluation of functions of $C([a,b],\bar H)$ 
in $t\in[a,b]$.

Our goal now is to replace $H$ by a quotient coalgebra without loosing principality. 
Using \cite[Example 2.29]{bhems}, it is straightforward to verify the following lemma.

\begin{lem}  \label{lem}
Let $H$ be a Hopf algebra with bijective antipode, 
let $\Delta_P: P\rightarrow P\ot H$ be a coaction making 
$P$ a right $H$-comodule algebra 
and $J$ a coideal right ideal of $H$. Then $C:= H/J$ is a coalgebra coacting on $P$ via 
$\rho_R := (\id\ot \pi)\circ \Delta_P$, \,$\pi:H\rightarrow C$ the canonical surjection, 
and the formula 
\[
\Psi : C\ot P\ni \bar \pi(h)\ot p \longmapsto p_{(0)} \ot \pi(hp_{(1)})\in P\ot C
\]
defines a bijective entwining making $P$ an entwined module. 
The inverse of $\Psi$ is given by 
\[ 
\Psi^{-1} (p\ot \pi(h))= \pi(h S^{-1}(p_{(1)}))\ot p_{(0)},
\]
and defines a left oaction on $P$ via 
\[ \label{inverse}
\rho_L: P\ni p \longmapsto \Psi^{-1}(p\ot \pi(1))=\pi(S^{-1}(p_{(1)} ) \ot p_{(0)} \in C\ot P.
\]
\end{lem}

\begin{lem} \label{princ}
Let $P$ be a principal $H$-comodule algebra 
for $\Delta_P: P\rightarrow P\ot H$. 
Also, let $J$ be a coideal right ideal of $H$ 
defining a coalgebra $C:= H/J$, let 
$\rho_R := (\id\ot \pi)\circ \Delta_P$, \,$\pi:H\rightarrow C$ the canonical surjection,  
be its right coaction on $P$, and let 
$i:C\rightarrow H$ be a unital (i.e., $i(\pi(1))=1$)    
$C$-bicolinear  (for the coactions $\Delta_H:= (\id\ot\pi)\circ\Delta$ 
and ${}_H\Delta:= (\pi\ot\id)\circ\Delta$) 
splitting (i.e., $\pi\circ i=\id$). 
Then $P$ is principal for the coaction $\rho_R$.
\end{lem}

\begin{proof}
Let $\ell:H\rightarrow P\ot P$ be a strong connection on $P$. 
One can easily check that $\ell\circ i:C\rightarrow P\ot P$ is a 
strong connection on $P$ for the right coaction 
$\rho_R := (\id\ot \pi)\circ \Delta_P$ and the left coaction 
$\rho_L := (\pi\ot\id)\circ {}_P\Delta$, where 
${}_P\Delta(p)= S^{-1}(p_{(1)})\ot p_{(0)}$ is the left $H$-coaction on $P$ 
viewed as a principal $H$-comodule algebra. 
On the other hand, it follows from Lemma \ref{lem} that 
$\Psi : C\ot P\ni \pi(h)\ot p \mapsto p_{(0)} \ot \pi(hp_{(1)})\in P\ot C$ 
is a bijective entwining making $P$ an entwined module. 
Therefore, since $\rho_R(1)=1\ot \pi(1)$, 
$\rho_L(1)= \pi(1)\ot 1$, and 
$\rho_L(p)= \Psi^{-1}(p\ot \pi(1))$ for all $p\in P$ 
by \eqref{inverse},
the principality of $P$ for the $C$-coaction $\rho_R$ 
follows from Lemma \ref{L1}.
\end{proof}

Combining Lemma~\ref{princ} with Theorem~\ref{prop-pe} 
yields the following result. 
\begin{thm} 
Let $\bar H$ be the $C^*$-algebra of a compact quantum group, 
$H$ its Peter-Weyl Hopf algebra, $J$ a coideal right ideal of $H$ 
and $\pi : H\rightarrow C:= H/J$ the cannonical surjection. 
Also let 
\begin{align}
&P_1:=\{f \in C([0,\mbox{$\frac{1}{2}$}],\bar H) \otimes H\,|\, (\mathrm{ev}_0\ot \id)(f) \in \Delta(H)\},\\
&P_2:=\{f \in C([\mbox{$\frac{1}{2}$},1],\bar H) \otimes H\,|\,
	(\mathrm{ev}_1\ot \id)(f) \in \mathbb{C} \otimes H\}
\end{align}
be right and left $C$-comodules for the right and left coactions 
\[
\rho^i_R:= (\id\ot\pi)\circ \Delta_{P_i}, \quad 
\rho^i_L:= (\pi\ot\id)\circ {}_{P_i}\Delta,\quad i=1,2,
\]
respectively. 
Here $\Delta_{P_i}:=\id\ot\Delta$ and 
${}_{P_i}\Delta:= (S^{-1}\ot\id)\circ\mathrm{flip}\circ \Delta_{P_i}$. 
Then, if there exists a unital bicolinear splitting $i: C\rightarrow H$ 
of $\pi : H\rightarrow C$, the pullback $C$-comodule 
\[
P:=\{ (p_1,p_2)\in P_1\times P_2\;|\; (\mathrm{ev}_{\frac{1}{2}} \ot\id)(p_1)
= (\mathrm{ev}_{\frac{1}{2}} \ot\id)(p_2)\}
\]
is principal.
\end{thm}

Let us now take a closer look 
at the compatibility of strong connections on 
principal comodules as appearing in the above theorem. 
First we observe that, if both of $\pi_1$ and $\pi_2$ 
defining the pullback diagram \eqref{dia} are surjective, 
then \eqref{key} 
simplifies to 
\[\label{ke}
\ell\;=\;\llp(\alpha^{12}_L+\id)\otimes(\alpha^{12}_R+\id)\lrp\circ\ell_2
+\;(\eta_1\circ\he-
m_{P_1}\circ(\alpha^{12}_L\otimes\alpha^{12}_R)\circ\ell_2)
*\big(\llp\id\otimes(\id+\alpha^{21}_R)\lrp\circ
\ell_1 
\big). 
\]
Indeed, since now $\ha_R^{21}$ is defined on the whole $P_1$, 
a special constructed connection 
$\tilde\ell_1$ in \eqref{elll} can be replaced 
by any strong connection $\ell_1$ on $P_1$. 
Note that specializing \eqref{ke} to comodule algebras 
coincides with what was obtained in \cite{hkmz}. 

Next, we observe that the formulae 
\begin{align}
&\ha^1 :   P_{12} \rightarrow P_1,\quad  \ha^1(\bar h\ot h)= 
2t\hs\bar h\ot h + (1-2t)\hs \bar\vare(\bar h)\hs h_{(1)}\ot h_{(2)},\\
&\ha^2 : P_{12} \rightarrow P_2,\quad \ha^2(\bar h\ot h)= 
2(1-t)\hs \bar h \ot h + (2t-1)\hs \bar\vare(\bar h) \ot h,
\end{align}
where $\bar\vare$ is any unital linear functional on $\bar H$,
define unital $C$-bicolinear splittings of $\pi_1$ and $\pi_2$, respectively. 
Hence we can take $\ha_L^{12}=\ha_R^{12}=\ha^{1}\circ \pi_2$ 
and $\ha_R^{21}=\ha^{2}\circ \pi_1$. 
Combining this with the fact that a cleaving map $j$ 
defines a strong connection via 
$\ell:=(j^{-1}\ot j)\circ \Delta$, we obtain very explicit formulae for 
strong connections on $P_1$ and $P_2$: 
\[
\ell_1:= (j_1^{-1}\ot j_1)\circ \Delta\circ i, \quad \ell_2:= (j_2^{-1}\ot j_2)\circ \Delta\circ i.
\]
Here $j_1:H\rightarrow P_1$, $j_1(h):=(t\mapsto h_{(1)})\ot h_{(2)}$, 
and $j_2:H\rightarrow P_2$, $j_2(h):=1\ot h$ are cleaving maps 
for $P_1$ and $P_2$, respectively.

\subsection{Quantum complex projective spaces $\mathbb{C}\mathrm{P}^3_{q,s}$}

Finally, we instantiate $\bar H$ to be $\CSUq$, \,$H=\SUq$, 
$J=(\cO(\mathrm{S}_{q,s}^2)\cap \mathrm{ker}\hs \vare)\SUq$, 
and $\bar\vare : \CSUq \rightarrow C$ to the counit. 
Here $\cO(\mathrm{S}_{q,s}^2)$ stands for the coordinate algebra 
of a Podle\'s quantum sphere $\mathrm{S}_{q,s}^2$, 
$s\in[0,1]$, \cite{BM}. 
(Note that the case $s=0$ brings us to the comodule-algebra 
setting.) 
The most interesting part of this structure is 
the unital bicolinear splitting of 
$\pi : \SUq\rightarrow \SUq/J$ given by \cite[Proposition~6.3]{BM}. 

All this defines a family of noncommutative deformations 
of the $\mathrm{U}(1)$-principal bundle 
\mbox{$\mathrm{S}^7\rightarrow \C\mathrm{P}^3$}. 
More precisely, we obtain deformations of a 
$\mathrm{U}(1)$-principal action on $\mathrm{S}^7$ 
given by 
\[
(z_1,z_2,z_3,z_4)\mathrm{e}^{\im \varphi}= 
(z_1\mathrm{e}^{\im \varphi},z_2\mathrm{e}^{-\im \varphi},
z_3\mathrm{e}^{\im \varphi},z_4\mathrm{e}^{-\im \varphi}), \quad 
|z_1|^2+|z_2|^2+|z_3|^2+|z_4|^2=1. 
\]
However, this is isomorphic with the diagonal action of 
$\mathrm{U}(1)$ on $\mathrm{S}^7$, so that the quotient space 
is again $\C\mathrm{P}^3$. 
Thus out of Pflaum's $\mathrm{S}^7$ we obtain a family of quantum 
projective spaces $\C\mathrm{P}^3_{q,s}$. 
A very explicit Mayer-Vietoris type formula for a strong connection on 
$\mathrm{S}^7_q\rightarrow \C\mathrm{P}^3_{q,s}$ 
should allow us to study the $K$-theory aspects of the tautological 
line bundle over $\C\mathrm{P}^3_{q,s}$, 
but this is beyond the scope of this paper.

\section*{Acknowledgements} The authors gratefully acknowledge financial
support from the following research grants: \linebreak[4]
PIRSES-GA-2008-230836,
1261\slash 7.PR~UE\slash 2009\slash 7, 
N201 1770 33,
and 189/6.PR~UE/2007/7. It is also a pleasure to thank Tomasz Brzezi\'nski
for helping us with Section~2, and Ulrich Kr\"{a}hmer for proofreading
the manuscript.

\end{document}